\documentclass[10pt,twoside,english]{amsart}
\normalsize
\advance\oddsidemargin by -1.0cm
\advance\evensidemargin by -1.0cm
\advance\topmargin by -1.0cm 
\usepackage{amssymb}
\usepackage{babel}
\usepackage{amsmath}
\usepackage{amscd}   
\usepackage{epsfig}  
\usepackage{rotating}
\usepackage{psfrag}
\usepackage[all]{xy}
\usepackage{multirow}
\usepackage{float}
\let\mathg\mathfrak
\theoremstyle{plain}
\newtheorem{cor}{Corollary}[section]
\newtheorem{lem}{Lemma}[section]
\newtheorem{thm}{Theorem}[section]
\newtheorem{prop}{Proposition}[section]

{}
\theoremstyle{definition}
\newtheorem{exa}{Example}[section]
\newtheorem{NB}{Remark}[section]

\newcommand{\bdm}{\begin{displaymath}}
\newcommand{\edm}{\end{displaymath}}
\newcommand{\be}{\begin{equation}}
\newcommand{\ee}{\end{equation}}
\newcommand{\ba}[1]{\begin{array}{#1}}
\newcommand{\ea}{\end{array}}

\newcommand{\btab}{\begin{tabular}}
\newcommand{\etab}{\end{tabular}}

\newcommand{\D}{\slash{\!\!\!\!D}}
\newcommand{\x}{\times}

\newcommand{\ox}{\otimes}

\newcommand{\ra}{\rightarrow}
\newcommand{\lra}{\longrightarrow}

\newcommand{\lmapsto}{\longmapsto}

\newcommand{\tr}{\ensuremath{\mathrm{tr}}}

\newcommand{\C}{\ensuremath{\mathbb{C}}}
\renewcommand{\H}{\ensuremath{\mathbb{H}}}
\newcommand{\R}{\ensuremath{\mathbb{R}}}
\newcommand{\K}{\ensuremath{\mathbb{K}}}

\newcommand{\Z}{\ensuremath{\mathbb{Z}}}

\renewcommand{\P}{\ensuremath{\mathbb{P}}}

\newcommand{\vrho}{\ensuremath{\varrho}}

\newcommand{\End}{\ensuremath{\mathrm{End}}}

\newcommand{\Ric}{\ensuremath{\mathrm{Ric}}}
\newcommand{\Scal}{\ensuremath{\mathrm{Scal}}}
\newcommand{\Scalg}{\ensuremath{\mathrm{Scal}^g}}

\newcommand{\Ad}{\ensuremath{\mathrm{Ad}\,}}

\newcommand{\diag}{\ensuremath{\mathrm{diag}}}

\newcommand{\gl}{\ensuremath{\mathg{gl}}}
\newcommand{\GL}{\ensuremath{\mathrm{GL}}}
\newcommand{\PSU}{\ensuremath{\mathrm{PSU}}}

\newcommand{\BSp}{\ensuremath{\mathrm{BSp}}}

\newcommand{\BSO}{\ensuremath{\mathrm{BSO}}}
\newcommand{\un}{\ensuremath{\mathg{u}}}
\newcommand{\su}{\ensuremath{\mathg{su}}}
\newcommand{\SU}{\ensuremath{\mathrm{SU}}}
\newcommand{\U}{\ensuremath{\mathrm{U}}}

\newcommand{\so}{\ensuremath{\mathg{so}}}
\newcommand{\SO}{\ensuremath{\mathrm{SO}}}
\renewcommand{\sp}{\ensuremath{\mathg{sp}}}
\newcommand{\Sp}{\ensuremath{\mathrm{Sp}}}
\newcommand{\Spin}{\ensuremath{\mathrm{Spin}}}

\newcommand{\Orth}{\ensuremath{\mathrm{O}}}
\newcommand{\p}{\ensuremath{\mathg{p}}}
\renewcommand{\k}{\ensuremath{\mathfrak{k}}}
\newcommand{\g}{\ensuremath{\mathfrak{g}}}
\newcommand{\h}{\ensuremath{\mathfrak{h}}}
\newcommand{\m}{\ensuremath{\mathfrak{m}}}
\newcommand{\z}{\ensuremath{\mathfrak{z}}}

\renewcommand{\a}{\ensuremath{\mathfrak{a}}}

\begin{document}
\def\haken{\mathbin{\hbox to 6pt{%
                 \vrule height0.4pt width5pt depth0pt
                 \kern-.4pt
                 \vrule height6pt width0.4pt depth0pt\hss}}}
    \let \hook\intprod
\setcounter{equation}{0}
%
%
\thispagestyle{empty}
%
\date{\today}
\title[$\Sp(3)$ structures]{\large $\Sp(3)$ \normalsize structures on \large 
$14$\normalsize-dimensional manifolds}
%
%
%
\author{Ilka Agricola}
\author{Thomas Friedrich}
\author{Jos H\"oll}
\address{\hspace{-5mm} 
Ilka Agricola, Jos H\"oll \newline
Fachbereich Mathematik und Informatik \newline
Philipps-Universit\"at Marburg\newline
Hans-Meerwein-Strasse \newline
D-35032 Marburg, Germany\newline
{\normalfont\ttfamily agricola@mathematik.uni-marburg.de}\newline
{\normalfont\ttfamily hoellj@mathematik.uni-marburg.de}}
\address{\hspace{-5mm} 
Thomas Friedrich\newline
Institut f\"ur Mathematik \newline
Humboldt-Universit\"at zu Berlin\newline
Sitz: WBC Adlershof\newline
D-10099 Berlin, Germany\newline
{\normalfont\ttfamily friedric@mathematik.hu-berlin.de}}
%
%
\keywords{Sp(3); rank two symmetric space; connection with skew-symmetric 
torsion; $G$ structures on Lie groups; isoparametric hypersurface.}  
\begin{abstract}
The present article investigates  $\Sp(3)$ structures on $14$-dimensional 
Riemannian manifolds, a continuation of the recent study of
manifolds modeled on rank two symmetric spaces (here: $\SU(6)/\Sp(3)$).
We derive topological criteria for the existence of such a structure and
construct large families of homogeneous examples. As a by-product,
we prove a general uniqueness criterion for characteristic connections
of $G$ structures and that the notions of biinvariant, canonical, and
characteristic connections coincide on Lie groups with biinvariant metric.
\end{abstract}
\maketitle
\pagestyle{headings}
%
%
%
\section{Introduction }\noindent
%
\subsection{Background}
The present article is a contribution to the investigation
of Riemannian manifolds modeled on rank two symmetric spaces,
carried out by different authors in recent years
(for example, \cite{Bobienski&N07}, \cite{Chiossi&F}, \cite{nur}, 
\cite{ABF}, \cite{Chiossi&M12}).
They constitute an interesting new class of special geometries
that goes back to Cartan's classical study of isoparametric hypersurfaces
(\cite{Cartan38}, \cite{Cartan39}), as we shall now explain.

A  Riemannian manifold  immersed in a space form with
codimension one is called
an  isoparametric hypersurface if its principal curvatures are constant;
the main case of interest are immersions into spheres $S^{n-1}\subset\R^n$, 
the case we shall be interested in henceforth.
If one denotes by $p$ the number of different principal curvatures,
Cartan proved that for $p=1,2$ only certain spheres are possible,
while for $p=3$, tubes of constant radius over an embedding of $\K\P^2$
into $S^{n-1}$ are possible for $\K=\R,\C,\H$, and $\mathbb{O}$:
Hence, for $p=3$, the dimension 
$n$ must be $5,8,14$, or $26$. The main key of the
construction are the so called Cartan-M\"unzner polynomials,
homogeneous harmonic polynomials $F$ of degree $p$ satisfying
$\|\mathrm{grad} F \|^2=p^2\|x\|^{2p-2}$. The level sets of $F\big|_{S^{n-1}}$
define an isoparametric hypersurface family. Geometrically,
$F$ can be understood as a symmetric rank $p$ tensor $\Upsilon$, and each level
set $M$ will be invariant under the stabilizer $G_\Upsilon$ of $\Upsilon$.
Hence, isoparametric hypersurfaces lead to Euclidean spaces $\R^n$ admitting
a symmetric rank $p$ tensor $\Upsilon$ and a $G_\Upsilon$ structure,
and, for $p=3$,  this leads us in a natural way to manifolds of
dimension $5,8,14$, and $26$.
 
The relation to rank two symmetric spaces is as follows:
If $M^{n-2}\subset S^{n-1}=\SO(n)/\SO(n-1)$ is orbit of some Lie group
$G\subset \SO(n)$, then it is automatically isoparametric.
Hence, the classification of homogeneous isoparametric hypersurfaces can
be deduced from the classification of all subgroups $G\subset \SO(n)$
such that the codimension in $S^{n-1}$ (resp.~$\R^n$) of its principal 
$G$-orbit is one (resp.~two). By results of Hsiang and Lawson,
this is exactly the case for the isotropy representations of
rank $2$ symmetric spaces \cite{Hsiang&L71}, \cite{Hsiang80}.
From the root data of the symmetric space, one deduces that for $p=3$,
only 4 symmetric spaces are possible, namely, 
$\SU(3)/ \SO(3)$, $\SU(3)$, $\SU(6)/\Sp(3)$, and  $E_6/F_4$.
Their relation to the division algebras  $\K=\R,\C,\H$, and $\mathbb{O}$
(see Cartan's result) is through their isotropy representations,
they are realized on trace free symmetric endomorphisms (see Table 
\ref{rk2-symm-spaces}).

\begin{table}
\bdm{\small
\ba{|l||l|l|l|l|}\hline
\text{dimension} & 5 & 8 & 14 & 26 \\ \hline 
\text{symmetric model} & \SU(3)/ \SO(3) & \SU(3)& \SU(6)/\Sp(3) &  E_6/F_4\\ \hline 
\text{isotropy rep.} & \SO(3) \text{ on } S^2_0(\R^3)& \SU(3) \text{ on }
S^2_0(\C^3) & \Sp(3) \text{ on }S^2_0(\H^3)   & E_6 \text{ on } S^2_0(\mathbb{O}^3) \\
\hline
\ea}\edm
\caption{Rank two symmetric spaces and their isotropy representations}
\label{rk2-symm-spaces}
\end{table}

We are interested in Riemannian manifolds in these $4$ dimensions
admitting a symmetric, trace free, $3$-tensor $\Upsilon$ \cite{nur};
its stabilizer is then resp.~$\SO(3),\ \SU(3), \Sp(3)$, or $F_4$.
The $5$-dimensional case and the corresponding $SO(3)$
structures were studied by several authors in \cite{ABF},
\cite{Bobienski&N07}, \cite{Chiossi&F}. For the $8$-dimensional case and 
the corresponding $\SU(3)$ structures, we refer to \cite{Hit}, \cite{Witt}, 
and \cite{puh}. The present paper will be the first dealing with $n=14$.
As far as we know, nothing is known for manifolds modeled on the
exceptional symmetric space $E_6/F_4$. From the experience
of the present work, one can expect the computations to be challenging, 
but this case has the charm
that it is the first occurrence of the exceptional Lie group $F_4$ in
differential geometry.

\subsection{Outline}
By definition,  an $\Sp(3)$ structure on a
$14$-dimensional Riemannian manifold will be a reduction of the 
frame bundle to an $\Sp(3)$-bundle. 
We  take a closer look at $\Sp(3)$ structures, and classify the different
types through their intrinsic torsion. This is the first occurrence 
where the high dimension implies the failure of standard techniques:
we were not able to prove the uniqueness of the so-called characteristic
connection of an $\Sp(3)$ structures in the usual way, and therefore
proved a general uniqueness criterion which is valuable in its own
(Theorem \ref{kernel-theta}), based on the skew holonomy theorem from
\cite{Agricola&F04} and \cite{Olmos&R08}.

We then derive some topological conditions for a $14$-dimensional 
manifold to carry an $\Sp(3)$ structure.
They are a consequence of the computation of the cohomology ring $H^*(BSp(3)\,
; \, \Z )=\Z[q_4,q_8,q_{12}]$ for some $q_i\in H^i(BSp(3))$, see \cite{mito}. In
particular, for a compact oriented Riemannian manifold with $\Sp(3)$ structure
the Euler characteristic as well as  the $i$-th
Stiefel-Whitney classes ($i\neq 4,8,12$) must vanish. 
Any $\Sp(3)$ structure on a $14$-dimensional
manifold induces a unique spin structure. Besides $\SU(6)/\Sp(3)$, we will
construct large families of manifolds admitting an $\Sp(3)$ reduction.

The next section is devoted to the existence problem of 
$\Sp(3)$ structures (and other $G$ structures) on Lie groups---for example,
whether $G_2$ carries an $\Sp(3)$ structure. For Lie groups equipped with
a biinvariant metric, we prove that the notions of characteristic, canonical,
and biinvariant connections coincide, and that these are precisely
the connections induced by the commutator (Theorem 
\ref{conn-on-lie-groups}). The link to
$\Sp(3)$ structure is subtle: Firstly, this result treats the case 
excluded in Theorem \ref{kernel-theta}; secondly, the result is
intricately linked to previous work by Laquer on
biinvariant connections \cite{Laquer92a}, \cite{Laquer92b}, in which
the rank two symmetric spaces and the Lie groups $\U(n),\ \SU(n)$
play an exceptional role.

The longest part of the paper is devoted to the explicit construction
and investigation of $14$-dimensional homogeneous manifolds with
$\Sp(3)$ structure, hence proving that such manifolds exist and that
they carry a rich
geometry. The manifolds are a higher dimensional analogue
of the Aloff-Wallach space, $\SU(4)/\SO(2)$, the related quotients 
$\U(4)/ \SO(2)\x \SO(2)$,
$\U(4)\x\U(1)/ \SO(2)\x \SO(2)\x\SO(2)$, and finally 
$\SU(5)/\Sp(2)$ (this is the
same manifold as the symmetric space $\SU(6)/\Sp(3)$, but the homogeneous
structure is different). In all situations, there are large families of
metrics admitting an $\Sp(3)$ structure with characteristic connection.
For the first three spaces, the qualitative
result is the following: the $\Sp(3)$ structure is of mixed type,
the characteristic torsion is parallel, and its holonomy is contained
in the maximal torus of $\Sp(3)$. For the last example, the picture
is different: It is a $3$-parameter deformation of the integrable
$\Sp(3)$ structure (i.\,e.~the structure corresponding to the symmetric
space), it is of mixed type for most metrics, but of pure
type for some, the characteristic connection has parallel torsion 
for a $2$-parameter subfamily, and its holonomy lies between
$\Sp(2)$ and $\Sp(3)$. The Appendix contains the explicit realizations of
representations needed for performing the calculations.
%
%
\subsection{Acknowledgments}
%
Some of the results presented in this article rely on involved 
representation theoretic
computations. These were obtained with help of the computer algebra
system LiE \cite{LiE}; hence, we thank Marc van Leeuwen (as representative of
the whole LiE team) for making such a nice tool available to the 
scientific community. The computer algebra system Maple was also
intensively used.
Some preliminary results of this article
appeared in the last author's diploma thesis \cite{Hoell11}.
The first author acknowledges financial support by the
 DFG within the priority programme 1388 "Representation theory".
The last author is funded through a Ph.\,D.~grant of
Philipps-Universit\"at Marburg.
%
\section{Definition and properties of   $\Sp(3)$ structures}\noindent
%
\subsection{Basic set-up}
The $14$-dimensional irreducible representation $V^{14}$ of the Lie group 
$\Sp(3)$ gives rise to an embedding $\Sp(3)\subset \SO(14)$. One possible 
realization of this representation is by conjugation on trace free
hermitian  quaternionic endomorphisms of $\H^3$, denoted by $S_0^2(\H^3)$. 
Therefore, it is natural to realize the Lie Group $Sp(3)$ as quaternionic, 
hermitian endomorphisms of $\H^3$: 
\bdm
\Sp(3) = \{g\in \SU(6) ~ | ~ g^t Jg=J\}=\{g\in
\GL(3,\H)~|~gg^t=\textbf{I}_3\}, 
\mbox{ where } J= \begin{bmatrix} 0 & \textbf{I}_3 \\ -\textbf{I}_3 & 0 \end{bmatrix}
\edm
and $\textbf{I}_3$ denotes the identity of $\C^3$ (respectively $\H^3$).
The second equality is established by 
$g=\begin{bmatrix}A&B\\-\bar{B}&\bar{A}\end{bmatrix}\mapsto A+jB$, 
$\{1,i,j,k\}$ being the usual quaternionic units. Thus we get the 
$\Sp(3)$-representation as
\bdm
\vrho(g)X\ :=\ gXg^{-1} \text{ for }g\in \Sp(3),\ X\in S^2_0(\H^3)\ \cong\ V^{14}.
\edm
We give a precise description of this representation in Appendix \ref{ch:ap:sp3}.
The space $S^2(\H^3)$ of symmetric quaternionic endomorphisms of 
$\H^3$ is a classical  Jordan algebra with respect to the product 
 $X\circ Y:=\frac{1}{2}(XY+YX)$. We define a symmetric $(3,0)$-tensor
 $\Upsilon$ by polarization from the trace,
\bdm
\Upsilon (X,Y,Z)\ :=\ 2\sqrt{3}[\tr X^3+ \tr Y^3+ \tr Z^3] - \tr(X+Y)^3 -
\tr(X+Z)^3-\tr(Y+Z)^3+\tr(X+Y+Z)^3.
\edm
A second tensor is obtained as $\tilde\Upsilon(X,Y,Z) := \Upsilon (\bar{X},
\bar{Y}, \bar{Z})$. Because of the non-commutativity of $\H$, the symmetric 
$(3,0)$-tensors $\Upsilon$ and $\tilde\Upsilon$ are not conjugate under the
action of $\SO(14)$, but they
both have stabilizer $\Sp(3)$. Alternatively, one may use the Jordan
determinant for defining a symmetric tensor; again, the  non-commutativity implies 
the existence of two determinants $\det_1, \det_2$. However, $\det_1(X)=\tr
X^3$, hence polarization and hermitian conjugation yields again the same tensors 
$\Upsilon$ and $\tilde\Upsilon$. We observe that, in this special situation,
there exists an  alternative object realizing the reduction from
$\SO(14)$ to $\Sp(3)$: $\Sp(3)$ is the stabilizer of a generic $5$-form
$\omega^5$ in $14$ dimensions. Thus, $\Sp(3)$ geometry continues in a 
natural way the investigation of $3$-forms ($n=7$ and $G=G_2$), and $4$-forms
($n=8$ and $G=\Spin(7)$ as well as all quaternionic K\"ahler geometries in
dimensions $4n$).

By definition, an $\Sp(3)$ structure on a $14$-dimensional Riemannian 
manifold $(M,g)$ is a
reduction of its frame bundle to a $\Sp(3)$ subbundle. This is 
equivalent to the existence of a $(3,0)$-tensor $\Upsilon$, which is to 
be associated with the linear map $TM\rightarrow \mathrm{End}(TM)$, 
$v\mapsto \Upsilon_v$ defined by $(\Upsilon_v)_{ij}=\Upsilon_{ijk}v_k$ with 
the following properties \cite{nur}
\begin{enumerate}
\item it is totally symmetric: $g(u,\Upsilon_v  w)= g(w,\Upsilon_v u)=
g(u,\Upsilon_w v)$,
\item it is trace-free: $\tr\Upsilon_v=0$,
\item it reconstructs the metric: $\Upsilon^2_v v =g(v,v)v$.
\end{enumerate}
A first example of such a manifold is the symmetric space $\SU(6)/\Sp(3)$. 
As Kerr shows in \cite[Section 4]{ker},  this is the space of quaternionic 
structures on $\R^{12}\cong\C^{6}$ for a fixed complex structure. Further non
symmetric examples will be given in Section \ref{ch:ex}.
\subsection{Types and general properties of $\Sp(3)$ 
structures}\label{subsection:types}
The different geometric types of $G$ structures on a Riemannian
manifold $(M,g)$, i.\,e.~of reductions $\mathcal{R}$ of the frame bundle
$\mathcal{F}(M)$ to the subgroup $G\subset \Orth(n)$,
are classified via the \emph{intrinsic torsion} (\cite{frie}, see
also \cite{Salamon89}, \cite{Fino98}).

Given a Riemannian $14$-manifold $M^{14}$ with an $\Sp(3)$ structure, 
we consider the Levi-Civita connection $Z^g$ as a $\so(14)$-valued 
$1$-form on the frame bundle $\mathcal{F}(M^{14})$. 
If unique, we shall denote the irreducible 
$\sp(3)$-representation
of dimension $n$ by $V^n$ (in particular, we shall write sometimes
$\sp(3)=V^{21}$). To start with,
the complement of the Lie algebra $\sp(3)$ inside $\so(14)$
is an irreducible $\sp(3)$-module $V^{70}$. 
Hence, the restriction of $Z^g$ 
to $\mathcal{R}$ can be split into
\bdm 
Z^g\big|_{T\mathcal{R} } \ = \ Z^* \, \oplus \, \Gamma \, \in \, \so(14)\ = \ \sp(3)\, \oplus\, V^{70},
\edm
where $\Gamma$ is called the \emph{intrinsic torsion}. In every point $x$,
$\Gamma_x\in  V^{14}\, \otimes \, V^{70}$. The following Lemma
may be checked directly with LiE:
\begin{lem}\label{decomp-reps}
 $\Lambda^3(V^{14})$ splits into four irreducible components,
\bdm
 \Lambda^3(V^{14})\ = \ \sp(3)\, \oplus \, V^{70} \, \oplus \, V^{84} \,
 \oplus \, V^{189} \ ,
\edm
and  $V^{14}\otimes V^{70}$ splits into seven irreducible components,
\bdm
 V^{14}\, \otimes \, V^{70} \ = \ \Lambda^3(V^{14})\, \oplus \, V^{14} \, \oplus
 \, V^{90} \, \oplus \, V^{512} \ .
\edm
\end{lem}
Thus, there are  $7$ basic types of $\Sp(3)$ structures,
classified by the irreducible submodules of $ V^{14}\, \otimes \, V^{70}$;
we call a structure \emph{of type $V^i$} if $\Gamma$ is contained in $V^i$
and we call it \emph{of mixed type} if $\Gamma$ is not contained in one
irreducible representation. 
Recall  that a given $\Sp(3)$ structures will admit an invariant metric
connection with skew symmetric torsion (`a' characteristic connection) 
if and only if $\Gamma$ lies in the image of the $\Sp(3)$-equivariant map \cite{frie}
\bdm
\Theta\, :=\, \mathrm{id}\otimes \mathrm{pr}_{V^{70}} : \quad 
\Lambda^3(V^{14})\lra  V^{14}\otimes V^{70}.
\edm
In this definition, we understand $\Lambda^3(V^{14})$ as a subspace
of $V^{14}\ox \Lambda^2(V^{14})$ and identify $\Lambda^2(V^{14})$ with
$\so(14)$. This shows  that $\Sp(3)$ structures with
$\Gamma\in  V^{14} \oplus V^{90} \oplus V^{512}$ cannot admit  a 
characteristic connection. The connection will be unique---and 
thus will deserve to be called
\emph{characteristic connection}---if and only if $\Theta$ is injective.
For small groups and dimensions, injectivity can often be checked directly,
and this is a well-known result for almost Hermitian or $G_2$ structures.
In our case, a direct verification fails for the first time; we will 
thus prove a general criterion
that follows from the skew holonomy Theorem of Olmos and Reggiani 
\cite{Olmos&R08}, based on preliminary work from our article
\cite{Agricola&F04}. Our result generalizes in some sense 
\cite[Thm 1.2]{Olmos&R08},
stating that the canonical connection of an irreducible naturally reductive
space ($\neq S^n,\, \R\P^n$ or a Lie group) is unique (i.\,e.~different
realizations as a naturally reductive
space induce the same canonical connection).
The case of  an adjoint representation
(excluded below) will be treated separately in Section \ref{ch:groups}.
\begin{thm}\label{kernel-theta}
Let $G\subsetneq \SO(n)$ be a connected Lie subgroup acting 
irreducibly on $\R^n$, and assume that $G$ does not act on $\R^n$ 
by its adjoint representation. 
Let $\m$ be a reductive complement of $\g$ inside $\so(n)$, 
$\so(n) = \g\oplus \m$. Consider the $G$-equivariant map 
\bdm
\Theta\, :=\, \mathrm{id}\otimes \mathrm{pr}_{\m} : \quad 
\Lambda^3(\R^n)\lra  \R^n\otimes \m, \quad
\Theta (T)\ =\ \sum_{i} e_i\ox \mathrm{pr}_{\m} (e_i\haken T).
\edm
Then  $\ker\Theta=\{0\}$, and hence the characteristic connection of a
$G$-structure on a Riemannian manifold $(M,g)$ is,
if existent, unique.
\end{thm}
\begin{proof}
An element $T\in\Lambda^3(\R^n)$ will be in $\ker\Theta$ if an only if
all $X\haken T$, identified with elements of $\so(n)$, lie in $\g$.
In the notation of \cite{Agricola&F04}, any $3$-form $T\in\Lambda^3(\R^n)$
generates a Lie algebra
\bdm
\g^*_T\ := \ \mathrm{Lie}\langle X\haken T \, |\, X\in \R^n \rangle,
\edm
and 
\bdm
\ker\Theta\ =\ T(\g,\R^n)\ :=\ \{T\in\Lambda^3(\R^n) \, |\,\g^*_T\subset \g\}.
\edm
In \cite{Olmos&R08}, a triple $(V,\theta,G)$ is called a 
\emph{skew holonomy system} if $V$ is an Euclidian vector space,
$G$ is a connected Lie subgroup of
$\SO(V)$, and $\theta: V\ra\g$ is a totally skew $1$-form
with values in $\g$, i.\,e.~$\theta(X)\in\g\subset\so(V)$ and
$\langle \theta(X)Y,Z\rangle $ defines
a $3$-form on $V$. Hence, we see that any $T\in\ker\Theta $  
defines a skew holonomy system with $V=\R^n$ and the given $G$ representation, 
and $\theta(X)= X\haken T$. Furthermore, this skew holonomy system
will be \emph{irreducible}, by assumption on the $G$-representation on $\R^n$.
By \cite{Agricola&F04}, \cite[Thm 4.1]{Olmos&R08}, $G$ cannot act transitively on the
unit sphere of $\R^n$, for then $G=\SO(V)$ would hold, and this case
was excluded by assumption. Thus, any $T\in\ker\Theta $ defines
a non-transitive irreducible skew holonomy system. By
the skew holonomy Theorem \cite[Thm 1.4]{Olmos&R08}, $\R^n$ will then
itself be a Lie algebra, with the bracket induced by $T$ 
($[X,Y]=T(X,Y,-)$), and $G=\Ad H$, where $H$ is the connected Lie group
 associated to the Lie algebra $\R^n$. This case having been
excluded by assumption, it follows that any $T\in\ker\Theta$ has to vanish.
\end{proof}
Let us look back at all $G$-structures modeled on the
four rank two symmetric spaces $SU(3)/SO(3)$, $\SU(3)$, $SU(6)/\Sp(3)$,
and $E_6/E_4$.  For the $5$-dimensional $\SO(3)$-representation, 
the injectivity of $\Theta$ can be established by elementary
methods \cite{frie}, \cite{ABF}. For $\SU(3)$, viewed as a symmetric space,
we are dealing with the adjoint representation excluded in
Theorem \ref{kernel-theta}, and indeed
the one-dimensional kernel of $\Theta$ was observed by Puhle in \cite{puh}.
For the irreducible representations of $\Sp(3)$ on $\R^{14}\cong V^{14}$ 
and  $F_4$ on $\R^{26}$,  Theorem 
\ref{kernel-theta} is applicable, hence $\ker \Theta=\{0\}$  and the
characteristic connection is unique in all situations where
at least one such connection exists. Together with the explicit decompositions
from Lemma \ref{decomp-reps}, we can summarize the result for
$\Sp(3)$-structures as follows:
\begin{cor}
An $\Sp(3)$ structure on a $14$-dimensional
Riemannian manifolds admits a  characteristic connection $\nabla^c$
if and only if the
$14$- , $90$- and $512$-dimensional parts of its 
intrinsic torsion vanish, and then it is unique.
\end{cor}
\begin{NB}
Even in cases where the $G$ action on $\R^n$ is not irreducible, a
modification of the proof of Theorem \ref{kernel-theta} might work. 
We leave it to the reader
to check this for example for the action of $\U(n)$ on $\R^{2n+1}$,
thus yielding the uniqueness (if existent) of a characteristic connection for
almost metric contact manifolds in all dimensions. Of course,
this was shown explicitely before in \cite{Friedrich&I1}.
\end{NB}
\begin{NB}
If the $\Sp(3)$-manifold $(M^{14},g)$ admits a characteristic connection
$\nabla^c$ with torsion $T\in\Lambda^3(M^{14})$, it satisfies 
$\nabla^c\Upsilon=0$ by the general holonomy principle. 
But for any $(3,0)$-tensor field $\Upsilon$,
\bdm
\nabla^c_V\Upsilon (X,Y,Z)\ =\ \nabla^g_V\Upsilon (X,Y,Z)
- \frac{1}{2}[\Upsilon \big(T(V,X),Y,Z\big) +\Upsilon \big(X, T(V,Y),Z\big) 
+ \Upsilon \big(X,Y, T(V,Z)\big)],
\edm
hence one concludes at once that $\nabla^c\Upsilon=0$ implies 
\be\label{nearly-int}
\nabla^g_V\Upsilon (V,V,V)\ = \ 0.
\ee
Such $\Sp(3)$-manifolds were called \emph{nearly integrable} by Nurowski
\cite{nur}, in analogy to nearly K\"ahler manifolds. However, one sees that
condition ($\ref{nearly-int}$) is not a restriction for the $\Sp(3)$ 
structure, making some of
the computations \cite[p.~11]{nur} unnecessary. In this paper,
we shall just speak of $\Sp(3)$ structures admitting a characteristic 
connection.
\end{NB}
\begin{lem}
Suppose that $(M^{14},g)$ is a Riemannian manifold with $\Sp(3)$-structure
admitting a characteristic connection $\nabla$ with torsion 
$T\in \Lambda^3(M^{14})$, and that the torsion is $\nabla$-parallel, $\nabla T=0$.
Then there exists a $\nabla$-parallel $2$-form $\Omega$.
\end{lem}
\begin{proof}
The symmetric $(3,0)$-tensor $\Upsilon$ induces by contraction a
$\nabla$-parallel vector field $\xi$,  hence
the $2$-form $\Omega := \xi\haken T$ will be $\nabla$-parallel as well.
\end{proof}
However, the $2$-form $\Omega$ will  be very
degenerate ($\Sp(7)$ is much larger than $\Sp(3)$) 
and thus it will not induce a symplectic structure. Observe
that $\Ric^\nabla$ will have vanishing eigenvalue in direction $\xi$.
\subsection{Topological constraints}
%
Let $\BSO(14)$ and $\BSp(3)$ be the classifying spaces of $\SO(14)$ 
and $\Sp(3)$ respectively. For a $14$-dimensional oriented Riemannian 
manifold $M^{14}$ we consider the classifying map of the frame bundle 
\bdm
  f:M^{14}\longrightarrow \BSO(14).
\edm
The existence of a topological $\Sp(3)$ structure is equivalent to the
existence of a lift $\tilde{f}$,  
\bdm
  \xymatrix{ &\BSp(3) \ar[d]^{\iota} \\
    M^{14} \ar[ur]^{\tilde{f}} \ar[r]^{f} & \BSO(14)}
\edm
Since the cohomology algebra of the space $\BSp(3)$ is generated 
by three elements in $H^4$, $H^8$ and $H^{12}$ (see Theorem 5.6.,  
Chapter III, \cite{mito}) we immediately obtain the following
\begin{thm}
If $M^{14}$ is a compact oriented Riemannian manifold with $\Sp(3)$ structure, then 
\begin{enumerate}
\item the Euler characteristic vanishes, $\chi(M^{14}) = 0$, \label{thm:cond:1}
\item $w_i(M^{14})=0$ for $i\neq 4,8,12$, where $w_i$ are the 
Stiefel-Whitney classes.\label{thm:cond:2}
\end{enumerate}
\end{thm}
\begin{NB}\label{nb:spin}
For example, $S^{14}$ and any product of spheres $S^n\x S^m$ with $m+n=14$
and $m,n$ both even cannot carry an $\Sp(3)$ structure, since the
Euler characteristic does not vanish.
The requirement  $w_1(M^{14})=w_2(M^{14})=0$ means that any manifold
with $\Sp(3)$ structure is orientable and admits a spin structure. 
Since $\Sp(3)$ is simply connected, the inclusion 
$\Sp(3)\subset\SO(14)$ admits a unique lift to $\Spin(14)$. Thus a 
$\Sp(3)$ structure defines a \emph{unique} spin structure.
\end{NB}
Now we are going to construct some  examples. In general, let $M^n$ be a
manifold with a fixed $G$-reduction $\mathcal{R}$ of the frame bundle. 
For any $G$-representation $\kappa$ on
$E_o$ we consider the associated bundle 
\bdm
 E\ = \ \mathcal{R}\times_\kappa E_0 \stackrel{\pi}{\rightarrow}M^n \ .
\edm
The tangent bundle of the manifold $E$ splits into a horizontal and vertical
part,
\bdm
T(E) \ = \ T^v(E) \, \oplus \, T^h(E) \ .
\edm
Since
\bdm
T^h(E) \ = \ \pi^*(T(M^n)) \ = \ \pi^*(\mathcal{R}) \times_G \R^n \, , \quad
\mbox{and} \quad
T^v(E) \ = \ \pi^*(E) \ = \ \pi^*(\mathcal{R}) \times_G E_o \,
\edm
we obtain the following
\begin{lem}\label{lem:vb}
 As a manifold, $E$ admits a $G$ structure $TE=\pi^*(R)\times_G(E_0\oplus\R^n)$.
\end{lem}
With Theorem \ref{thm:maxsub} in Appendix \ref{ch:ApSubGroups} we get 
$V^{14}\ = \ \Delta_5\, \oplus\, \p^5\, \oplus\, \p^1$ and thus receive
\begin{cor}
Any oriented $1$-dimensional bundle over a $13$-dimensional Riemannian
manifold with an $\Sp(2)$ structure of type $\Delta_5 \oplus \p^5$ admits a
$\Sp(2) \subset \Sp(3)$ structure.
\end{cor}
Let us consider a $5$-dimensional manifold $M^5$ with $\Spin(5)\cong\Sp(2)$
structure as well as  the corresponding spinor bundle. It is a $13$-dimensional
manifold and  Lemma \ref{lem:vb} gives us the needed $\Sp(2)$ structure on it.
We summarize the result, 
\begin{exa}
 Any oriented $1$-dimensional bundle over the spinor bundle of a
 $5$-dimensional 
spin manifold admits a $\Sp(2)\subset\Sp(3)$ structure.
\end{exa}
Taking a $8$-dimensional manifold with a $\Spin(5) = \Sp(2)$ structure
$\mathcal{R}$ we consider $M^{13}=\mathcal{R}\times_{\Sp(2)}\p^5$. Again we
have $T(M^{13})= \mathcal{R}\times_{\Sp(2)}(\Delta_5\oplus\p^5)$, leading to a 
$\Sp(3)$ structure on any $S^1$ bundle over $M^{13}$.
\begin{exa}
 Any oriented $1$-dimensional bundle over the associated bundle 
$M^{13}=\mathcal{R}\times_{\Sp(2)}\p^5$ of a
 $8$-dimensional manifold $X^8$ 
with $\Sp(2)$ structure $\mathcal{R} \ra X^8$ admits a $\Sp(2)\subset\Sp(3)$ 
structure.
\end{exa}
The above examples of $\Sp(3)$ spaces being fibrations over special 
smaller dimensional manifolds used the subgroup $G=\Sp(2)\subset\Sp(3)$ 
as well as its decomposition of $V^{14}$. We list the maximal connected 
subgroups of $\Sp(3)$ and their decompositions of $V^{14}$ in Appendix 
\ref{ch:ApSubGroups}. Particularly interesting are $G=\U(3)$ and $G=SO(3)$.
\begin{exa}
 A special $9$-dimensional real vector bundle over a $5$-dimensional 
manifold equipped with an irreducible $\SO(3)$ structure admits a 
$\SO(3)\subset\Sp(3)$ structure (see \cite{ABF}).
\end{exa}
\begin{exa}
 A special $8$-dimensional real vector bundle over a $6$-dimensional 
hermitian manifold admits a $\U(3)\subset\Sp(3)$ structure.
\end{exa}
%
%
\section{$\Sp(3)$ structures and other $G$ structures on Lie groups}
\label{ch:groups}
%
The first $14$-dimensional homogeneous space that comes to  mind 
(besides $S^{14}$) is presumably the Lie group $G_2$. We will devote
this section to the question whether  $G_2$ carries a natural
$\Sp(3)$ structure. Since it seems that the 
topic has not been treated before,
we shall start with some general comments on $G$ structures on Lie groups.

Let $G$ be a connected compact Lie group with a biinvariant metric $g$,
and $K\subset G$ a connected subgroup of $G$ whose Lie algebra $\k$
decomposes into center $\z$ and simple ideals $\k_i$, 
i.\,e.~$\k=\z\oplus \k_0\oplus\ldots\oplus \k_r$. 
Set $\a:=\k^\perp$, hence $\g=\a\oplus \k$. We view $G$
as the homogeneous space $G\x K/\Delta(K)$, where
$\Delta K := \{(k,k)\, : \ k\in K\}$. D'Atri and Ziller  proved
in \cite[p.~9]{DAtri-Ziller79} that the family of left invariant metrics 
on $G$ defined by ($\alpha,\alpha_1,\ldots, \alpha_r>0$, $h$ 
any scalar product on $\z$)
\be\label{inv-metrics}
\langle \, ,\, \rangle\ :=\ \alpha \cdot g\big|_{\a} \oplus
h\big|_{\z}\oplus \alpha_1\cdot g\big|_{\k_1}\oplus \ldots\oplus
\alpha_r\cdot g\big|_{\k_r}
\ee
is naturally reductive for the homogeneous space $G\x K/\Delta(K)$ in the following sense: Write $\g\oplus\k = \Delta \k \oplus \p$, then $\p$ is 
isomorphic (as a vector space) to $T_e (G\x K/\Delta(K))\cong \g$,
but with an isomorphism (and thus a commutator) 
depending on the parameters $\alpha,\alpha_i$. The metric then satisfies
\bdm
\langle [X,Y]_\p,Z\rangle + \langle Y, [X,Z]_\p\rangle \ =\ 0 \  \ \ 
\forall\  X,Y,Z\in\p.
\edm
In the special case that $K$ is chosen to be $G$, $\a=0$ and the metric
$\langle \, ,\, \rangle$ is precisely a biinvariant metric on $G$. 

By a theorem of Wang \cite{KN1}, invariant metric
connections $\nabla$ on $G$ (still with respect to its realization
as a naturally reductive space) are in bijective 
correspondence with linear maps
$\Lambda: \, {\p}\ra\so(\p)$ that are equivariant under 
the isotropy representation. As described in 
\cite[Lemma 2.1, Dfn 2.1]{Agricola03},
the torsion $T$ of $\nabla$ will be totally skew-symmetric
if and only if $\Lambda(X)X=0$ for all $X\in\p$ (and this
condition is well-known to be equivalent to the fact that
the geodesics of $\nabla$ coincide with the geodesics of the canonical connection, \cite[Prop. 2.9, Ch.X]{KN2}). Thus, one can give immediately
a one-parameter family of invariant metric connections $\nabla^t$ 
with skew torsion, namely the one defined by $\Lambda(X)Y=t [X,Y]_\p$
that was investigated in detail in \cite{Agricola03}. For $t=0$,
$\nabla^t$ has holonomy $K$, thus we can summarize: 
\begin{prop}
Let $G$ be a connected compact Lie group, $K\subset G$ a connected subgroup,
$G\cong G\x K/\Delta K$ as a reductive homogeneous space. 
For any parameters $\alpha,\alpha_1,\ldots, \alpha_r>0$,
the left invariant metric $\langle \, ,\, \rangle$  on $G$ defined by
$(\ref{inv-metrics})$ is naturally reductive and admits an 
invariant metric connection  with skew torsion and holonomy $K$. 
\end{prop}
In general, this
is all  we can say; in particular, we do not know about other systematic
constructions of interesting $K$ structures on a Lie group $G$, for example,
if $K$ is not a subgroup of $G$. 
 We shall now investigate further the case $K=G$.
First, we can answer the question on $G_2$ we started with:
\begin{NB}
Since $G_2$ is simple, there are no center nor non-trivial ideals that would allow
for a deformation of the metric, hence $\langle \, ,\, \rangle$ has to be
a multiple of the negative of the Killing form of $G_2$.
$\Sp(3)$ is not a subgroup of $G_2$, but its maximal \emph{simple} subgroup 
$\SU(3)$ (see Appendix \ref{ch:ApSubGroups}) is also a  maximal 
subgroup of $G_2$. However, they  are not conjugate inside $\SO(14)$; this is easiest
seen by computing the branching of the $14$-dimensional representation of $\Sp(3)$
resp.~$G_2$ to their resp.~subgroups  $\SU(3)$; It turns out that
these do not coincide. The smaller subgroups do not seem to be very interesting.
We conclude that \emph{$G_2$ does not carry an $\SU(3)\subset \Sp(3)$ structure
of the type described before}. 
\end{NB}
Going back to the general case $K=G$, we are now in the situation that
$\langle \, ,\, \rangle$ is  a biinvariant metric on $G$;
an affine connection $\nabla$ is a bilinear map $\Lambda:\ \g\ra \gl(\g)$,
with $\nabla_X Y = \Lambda(X)Y$. Alternatively, it is sometimes
more useful to formulate the properties of $\Lambda$ in terms of its
dual bilinear map $\lambda:\ \g\x\g\ra\g, \ \lambda(X,Y):=\Lambda(X)Y$.
\subsection*{Characteristic vs. canonical vs. biinvariant connections on
Lie groups}
We begin by clarifying the different notions of `interesting' connections
on compact connected  Lie groups (still with a biinvariant 
metric) and their relations.

In \cite{Laquer92a}, Laquer defined a \emph{biinvariant connection} 
on a Lie group
$G$ as any $(G\x G)$-invariant connection on the symmetric space 
$G\x G/ \Delta G$, as described through Wang's Theorem in \cite{KN2}. 
The connection $\nabla^\lambda$
defined by a bilinear map  $\lambda:\ \g\x\g\ra\g$ will be biinvariant
if and only if \cite[Thm 6.1]{Laquer92a}
\bdm
\lambda  (\Ad_gX,\Ad_g Y)\  =\ \Ad_g\lambda(X,Y)\quad\forall g\in G .
\edm
Alternatively, one often uses the adjoint map 
$\Lambda: \g\ra \End\,\g,\ \Lambda_X (Y)=\lambda(X,Y) $,
for which the property then reads $\Lambda_{\Ad_g X} = \Ad_g
\Lambda_X \Ad_g^{-1}$; of course, $\Lambda$ is just the map used in
Wang' Theorem.
Observe that the set of biinvariant connections forms a vector space,
so uniqueness is to be expected at best up to a scalar, and that the notion
does not depend on the metric. Evidently, $\lambda(X,Y)=c [X,Y]$ is
always a biinvariant connection.
\begin{lem}
The following conditions for a biinvariant connection $\nabla^\lambda$,
on a Lie group $(G,g)$ with biinvariant metric are equivalent:
\begin{enumerate}
\item $\Lambda_V\in\so(\g)$ for any $V\in\g$, i.\,e.
$g(\Lambda_V X,Y)+ g(X,\Lambda_V Y)\ =\ 0$;
\item $\nabla^\lambda$ is metric; 
\item The torsion $T^\lambda(X,Y,Z)$ of $\nabla^\lambda$ 
is skew symmetric, i.\,e.~$T\in\Lambda^3(\g)$.
\end{enumerate}
\end{lem}
\begin{proof}
The equivalence of (1) and (2) is immediate (for any metric).
One checks that $\nabla^\lambda$ has torsion and curvature transformation
\bdm
T^\lambda(X,Y)\ =\ \lambda(X,Y)-\lambda(Y,X)-[X,Y],\quad
R^\lambda(X,Y)\ =\ [\Lambda_X,\Lambda_Y]- \Lambda_{[X,Y]},
\edm
which shows the equivalence of (1) and (3) for biinvariant metrics.
Observe that $\Lambda_V$ will not, in general, be a representation;
rather, the second formula shows that this is equivalent to $\nabla^\lambda$
being flat.
\end{proof}
On the other hand, the \emph{canonical connection} $\nabla^c$ of a reductive
homogeneous space $M=\tilde G/\tilde K$ is, by definition, the unique 
connection induced from the $\tilde K$-principal fibre bundle 
$\tilde G\ra \tilde G/\tilde K$ (alternatively:
the $\nabla^c$-parallel tensors are exactly the $\tilde G$-invariant ones).
A priori, it depends on the choice of a reductive complement $\tilde\m$ of
$\tilde\k$ in $\tilde\g$, since it has torsion $T^c(X,Y)=-[X,Y]_{\tilde\m}$; 
for a Lie group (i.\,e. $\tilde G= G\x G$, 
$\tilde K = \Delta G$), it turns out that, unlike for naturally reductive 
spaces (cf.~Section \ref{subsection:types} and
the comments to Thm 2.1), each choice of a complement of
$\Delta \g\subset \g\oplus\g$
induces a different
canonical connection. The easiest way to see this is to construct (some of)
them explicitely. One checks that every space ($t\in\R$)
\bdm
\m_t\  :=\ \{X_t :=(tX, (t-1)X)\in \g\oplus\g\, :\ X\in\g \}\ \cong\ \g
\edm
defines a reductive complement of $\Delta \g$. One then computes
the decomposition of any commutator $[X_t,Y_t]$ in its $\Delta \g$- and
$\m_t$-part,
\bdm
[X_t,Y_t]\ =\ (t^2[X,Y], (t-1)^2[X,Y])\ =\ 
(t^2-t) ([X,Y], [X,Y]) + (2t-1)(t[X,Y], (t-1)[X,Y]),
\edm
Thus, the torsion $T^c(X,Y)$ becomes, after identifying $\m_t$ with $\g$
in the obvious way,
\be\label{inv-torsion}
T^c(X,Y)\ =\ -[X_t,Y_t]_{\m_t}\ =\ (1-2t)[X,Y].
\ee
For a biinvariant metric $g$, $T^c\in\Lambda^3(\g)$ (and this is equivalent
to the property that $\nabla^c$ is metric); $t=1/2$ corresponds
to the Levi-Civita connection, while $t=0,1$ are the flat $\pm$-connections
introduced by Cartan and Schouten (see \cite{Agricola&F10} and 
\cite{Reggiani10}). The holonomy of these connections is either
trivial ($t=0,1$) or $G$ ($t\neq 0,1$). The corresponding map
$\lambda^c:\ \g\x\g\ra\g$ is $\lambda^c(X,Y)=(1-t)[X,Y]$.

Finally, let us describe characteristic connections on Lie groups---i.\,e.~we
take  $M=G$ with a biinvariant metric and
consider $G\subset \SO(\g)$ through the adjoint representation. 
Again, we identify $\so(\g)\cong\Lambda^2\g$ and decompose it
under the action of $G$ into the representations 
$\so(\g)=\g\oplus\m$ where,
in general, $\m$ will not be irreducible.
The crucial observation  is that the intrinsic torsion 
$\Gamma$ (Section \ref{subsection:types} and 
\cite{frie}) vanishes, $\Gamma=0$, because the Levi-Civita connection is
a $(G\x G)$-invariant connection on $G$. 
Hence, $\ker\Theta\subset \Lambda^3(\g)$
parameterizes the space of characteristic connections (recall
that a characteristic connection is metric with skew torsion by construction).

Suppose $G$ is a compact connected Lie group. Its Lie algebra
splits into $\g=\z\oplus \g_1\oplus\ldots\oplus\g_q$, where
$\z$  is the center  and 
the $\g_i$ are the simple ideals in $\g$. The one-parameter family of
connections with torsion given by $(\ref{inv-torsion})$ has then an obvious 
generalization to a $q$-parameter family by rescaling the commutator
separately in all simple ideals,
\be\label{general-torsion}
T(X,Y)\ =\ \sum_{i=1}^q \alpha_i [X,Y] \big|_{\g_i},\quad
\alpha_i\in\R.
\ee
This connection is certainly biinvariant; it is also a canonical
connection for the reductive space $G\x G/\Delta G$, because  the
complement of $\Delta \g\subset \g\oplus\g$ can be chosen with a 
different parameter in each simple ideal $\g_i$. 
Finally, it also lies in $\ker \Theta$.
\begin{thm}\label{conn-on-lie-groups}
For a compact connected  Lie group $G$ with a biinvariant metric $g$,
the following families of connections coincide:
\begin{enumerate}
\item metric biinvariant connections with skew torsion,
\item metric canonical connections with skew torsion of the reductive 
spaces $G\x G/\Delta G$,
\item characteristic connections,
\end{enumerate}
and there is exactly one family of connections
with these properties, namely the one defined by eq. $(\ref{general-torsion})$.
\end{thm}
\begin{proof}

Consider a linear map $0\neq \lambda:\ \g\x\g\ra\g$ defining a biinvariant connection.
We interpret $\lambda$ as an intertwining
map $\mu: \g\ox\g\ra\g$ for $\Ad G$ by setting 
$\mu(X\ox Y):=\lambda(X,Y)$; the intertwining property is exactly the
biinvariance condition. Hence, the interesting question is to find copies
of $\g$ inside $\g\ox\g$. It is well-known that $\g\ox\g$ splits into the 
$G$-modules $\g\ox\g=S^2\g\oplus \Lambda^2\g$, and that $\g$ will always be
a submodule of  $\Lambda^2\g$ (however, there are also compact Lie groups
for which  $\g$ appears in $S^2\g$ as we will discuss later).
Decompose $\mu$ into its symmetric and antisymmetric part,
$\mu=\mu^s+\mu^a,\ \mu^s:\ S^2\g\ra\g,\ \mu^a:\ \Lambda^2\g\ra\g$.

\medskip
\emph{1st case: $\mu^s=0$}, i.\,e. $\mu=\mu^a$ is antisymmetric. This is 
the generic case that one would expect, as it includes the family of 
connections
that we already constructed. Since we're only interested in $\mu\neq 0$,
its dual map $\tilde\mu: \g\ra\g\subset\Lambda^2\g\cong\so(\g)$ exists.
The torsion of the connection defined by $\mu$ is
$T(X,Y,Z)=2g(\mu(X\ox Y),Z)- g([X,Y],Z)$; by assumption, it is a 
$3$-form in $\ker\Theta$, and since we knew before that $g([X,Y],Z)\in
\ker\Theta$, we can conclude that $g(\mu(X\ox Y),Z)\in\ker\Theta$ as well.
This proves that any biinvariant connection is characteristic in the sense
described before. One checks that the argument can be inverted, 
hence the sets of antisymmetric biinvariant connections and characteristic 
connection coincide.

Consider now a canonical connection, i.\,e.~the connection induced by
the choice of a reductive complement $\m$ of $\Delta\g\subset\g\oplus\g$.
It is tedious to describe all possible spaces $\m$; happily  it turns out
not be necessary (as a remark, we note that different complements
 will not necessarily induce different connections, for example, the embedding
of the center has no influence on the connection).
Whatever $\m$, the torsion $T^c_\m(X,Y)=-[X,Y]_\m$ of its canonical connection
is an $\Ad G$-equivariant antisymmetric map $\Lambda^2\g \ra\g$ and hence 
defines a biinvariant connection with $\mu^s=0$. It is a priori not clear 
whether one can find to any biinvariant connection satisfying $\mu^s=0$ 
a reductive complement $\m$ such that it coincides with its canonical
connection, but we will not need this.
 
\medskip
\emph{2nd case: $\mu^s\neq 0$}. We wish to exclude this case.
Unfortunately, we have to apply a `brute force' argument.
Recall that we showed that the connections (2) and (3) are (special)
biinvariant connections; hence it suffices to prove that
a metric biinvariant connection with skew torsion has necessarily $\mu^s=0$.
We will use the classification of biinvariant connections of compact Lie
groups given by Laquer. In \cite[Table I]{Laquer92a},
he decomposed the $G$-representations $S^2\g$ and $\Lambda^2\g$ for all 
compact simple Lie groups. He confirmed that $\g$ appears with
multiplicity one in $\Lambda^2\g\subset\g\ox\g$ for all of them,
but furthermore, he obtained the surprising result that $\g$ 
does not occur in $S^2\g $ for all of them -- except for
 $G=\SU(n),\, n\geq 3$ (which includes $\SO(6)$, since 
$\so(6)\cong\su(4)$). The Lie group $G=\SU(n)$ has a copy of $\su(n)$
in $S^2\g$ as well, corresponding to a symmetric $\Ad\SU(n)$-equivariant
map $\eta:\ \su(n)\x\su(n)\ra\su(n)$ given by \cite[p.550]{Laquer92a}
\be
\eta(X,Y)\ =\ i \alpha \left[XY+YX- \frac{2}{n}\tr(XY)\cdot I \right],\quad
\alpha\in\R.
\ee
A biinvariant metric on $G=\SU(n)$ is necessarily a multiple of the 
negative of the Killing form, hence we can take $g(X,Y)=-2n\,\tr(XY)$.
An elementary computation shows that, for general $X,Y,Z\in\su(n)$,
the quantity $g(\eta(X,Y), Z)+ g(\eta(X,Z),Y)\neq 0$, hence
the biinvariant connection defined by $\eta$ is \emph{not} metric and thus
not of relevance for us.
In fact, Laquer  himself extended his result to \emph{arbitrary} 
compact Lie groups  \cite[Thm 10.1]{Laquer92a}. 
The result is similar, if slightly more involved.
Besides $\SU(n)$, only $\U(n)$ admits symmetric maps 
$\eta:\ \un(n)\x\un(n)\ra\un(n)$; they span a $3$-dimensional space for
$n=2$ and a $4$-dimensional space for $n\geq 3$, and there
is an additional antisymmetric map (besides the obvious one $[X,Y]$),
namely $\nu(X,Y)=i (X\tr Y - Y\tr X)$. Using that a biinvariant metric on 
$\U(n)$ is just any positive definite extension to the center
of the metric of $\SU(n)$, one checks that none of them yields a metric 
connection. 

\smallskip
All in all, only connections corresponding to the embedding of
$\g$ inside $\Lambda^2 \g$ are candidates for all three types of
connections, and these are of course the ones described by
eq. $(\ref{general-torsion})$. This finishes the proof.

\smallskip
Although not necessary, we give an alternative proof
of the claim for $\g$ semisimple (assuming that one already
established that the connections (1) and (3) coincide and that they include the
connections (2)) --- it has the charm that it
does not need the classification results of Laquer. However, we
were not able to extend it to the compact case without using
the classification, so it does not improve the situation much.

We begin with the case that $G$ is simple, hence
the adjoint representation is irreducible.  As explained in 
the proof of Theorem \ref{kernel-theta}, any $T\in \ker\Theta$ 
defines then an irreducible  skew holonomy system $(V=\g,\theta,G)$, 
and for dimensional reasons, $G\neq \SO(\g)$, so the system is non-transitive.
By \cite[Thm 2.4]{Olmos&R08}, an irreducible non-transitive skew holonomy 
system is symmetric, and  the map $\theta: 
\g\ra\g\subset\so(\g)\cong\Lambda^2(\g)$ is unique up to a
scalar multiple \cite[Prop. 2.5]{Olmos&R08}---namely, it is
given by $\theta(X)=[X,-]\in\Lambda^2(\g) $, and $\theta(X)=X\haken T$. 
We conclude that for $G$ simple,
the space of characteristic connections is indeed given by
$(\ref{inv-torsion})$.  By the splitting theorem (see 
\cite[Section 4]{Agricola&F04}
or \cite[Lemma 2.2]{Olmos&R08}), the conclusion still holds for $G$ semisimple.
\end{proof}
\begin{NB}
The three rank two symmetric spaces $\SU(3)/\SO(3)$, $\SU(6)/\Sp(3)$,
and $E_6/E_4$ have a remarkable connection property very similar
to the one described for the Lie groups $\SU(n),\, \U(n)$ in the proof above,
again due to Laquer. 
In \cite{Laquer92b}, it was observed that
the symmetric spaces  $\SU(n)/\SO(n)$, 
$\SU(2n)/\Sp(n)$, and $E_6/E_4$ admit invariant affine connections
that are \emph{not} induced from the commutator. However,
a closer inspection of the three symmetric spaces yields that
these connections are not metric with skew-symmetric torsion, and hence
do not yield further candidates for a characteristic connection,
in agreement with the uniqueness statement from 
Theorem \ref{kernel-theta}. We were not
able to relate the existence of these `exotic' connections  to the 
characteristic connection or any other deeper geometric property of geometries
modeled on rank 2 symmetric spaces; but it may be worth mentioning
that their existence is linked to  the existence of a Jordan product.
\end{NB}
%

\section{The geometry of some homogeneous $\Sp(3)$ structures}\label{ch:ex}
%
Since there are no Lie groups carrying any reasonable $\Sp(3)$ structure,
it is natural to ask for homogeneous spaces with such a structure.
This section is devoted to the explicit construction of some 
$14$-dimensional homogeneous 
spaces carrying $\Sp(3)$-structures and their geometric properties. 

We choose a reductive complement $\m$ of $\sp(3)$ inside $\su(6)$, 
$\su(6)\cong\m\oplus\sp(3)$; an explicit realization
as well as a description of the  $14$-dimensional isotropy 
representation $\vrho(\sp(3))\subset\so(\m)\cong\so(14)$ of $\Sp(3)$
 is being given in Appendix \ref{ch:ap:sp3}.
The notation will be as follows:
The homogeneous spaces will be realized as quotients $M_i=K_i/H_i$ for a
running index $i$, yielding at Lie algebra level the reductive decompositions
\bdm
\k_i \ \cong \ \m_i \ \oplus \ \h_i \ \cong \
\langle K^i_j~|~j=1..14\rangle \ \oplus \ \langle H^i_j~|~j=1..r_i\rangle  .
\edm
Again, the explicit elements $K^i_j$ and $H^i_j$ will be listed in the
Appendix for each example. 
We will show that we can identify the subspaces $\m\cong\m_i$ inducing
$\vrho_i(\h_i)\subset\vrho(\sp(3))$ and, consequently, the $H_i$ structure
is a reduction of an $Sp(3)$ structure. 
%
\subsection{The higher dimensional Aloff Wallach manifold
$\SU(4)/\SO(2)$}
\label{ex:1}
We embed $H_1=\SO(2)$ in the Lie group $K_1=\SU(4)$ as
\bdm
SO(2)\ni \begin{bmatrix} \cos t & \sin t \\ -\sin t&\cos t \end{bmatrix} 
\lmapsto \diag(e^{-it},e^{-it},e^{it},e^{it})\in \SU(4).
\edm
The action of $\h_1=\so(2)$ on the $14$-dimensional $\Sp(3)$-representation 
$V^{14}$ splits into four $2$-dimensional representations and six trivial 
ones. For an invariant metric, we choose multiples of the Killing form on 
the invariant spaces parameterized by coefficients $\alpha, \alpha_2,\ldots >0$,  
\bdm g^{\alpha,..,\gamma}\ =\ \diag(\alpha,\alpha,\alpha_{2},\alpha_{2},
\alpha_{3},\alpha_{3},\alpha_{4},\alpha_{4},\alpha_{5},\alpha_{6},
\alpha_{7},\alpha_{8},\beta,\gamma)
\edm
with respect to the  orthonormal basis $K^1_1,\ldots, K^1_{14}$ of $\m_2$
described in Appendix \ref{ch:ap:su4}. The justification for this choice
of notation stems from the following result.
\begin{thm}
Consider the manifold $M_1=\SU(4)/\SO(2)$ equipped with the metric
$g^{\alpha,..,\gamma}$. 
For any parameters $\alpha,\alpha_i,\beta,\gamma>0$, it carries
an  $98$-dimensional space of invariant $Sp(3)$-connections,
and for $\alpha=\alpha_2=..=\alpha_8$, the $Sp(3)$ structure
admits a characteristic connection with torsion 
$T^{\alpha\beta\gamma}\in\Lambda^3(M_1)$. These $Sp(3)$ structures
with characteristic connection have the following properties:
\begin{enumerate}
\item The  characteristic connection has alwas parallel torsion,
$\nabla^{\alpha\beta\gamma}T^{\alpha\beta\gamma}=0$.
\item The structure is of mixed type, 
$T^{\alpha\beta\gamma}\notin V^i$ for  $i=21,70,84,189$.
\item The structure is never integrable, i.\,e.~there are no parameters 
$\alpha,\beta,\gamma>0$ with vanishing torsion.
\item For the characteristic connection $\nabla^{\alpha\beta\gamma}$, 
  the Lie algebra of the holonomy group is a subalgebra of the maximal 
torus of $\sp(3)$ and it is
\begin{itemize}
\item one-dimensional if $\beta=\alpha=\gamma$,
\item two-dimensional if ($\beta=\alpha$ and $\gamma\neq\alpha$) or 
($\beta\neq\alpha$ and $\gamma=\alpha$),
\item three-dimensional if $\beta\neq\alpha\neq\gamma$.
\end{itemize}
\end{enumerate}
\end{thm}
\begin{proof}
In Appendix \ref{ch:ap:su4} we construct a decomposition of the relevant 
Lie algebras. By a theorem of Wang \cite{KN1}, invariant metric
connections $\nabla^{\alpha,..,\gamma}$ are in bijective 
correspondence with linear maps
$\Lambda_{\m_1}: \, {\m_1}\ra\so(\m_1)$ that are equivariant under 
the isotropy representation $\vrho_1$,
\bdm
\Lambda_{\m_1} (\vrho_1(h)X)\ =\ \vrho_1(h)\Lambda_{\m_1}(X)\vrho_1(h)^{-1} \quad
\forall h\in \SO(2),\ X\in \m_1.
\edm
A connection is an $Sp(3)$ connection  if the image of $\Lambda_{\m_1}$ is 
inside $\sp(3)$, $\Lambda_{\m_1}:  {\m_1}\ra\sp(3)$.
One calculates all such maps $\Lambda_{\m_1}$. They are given by the two 
following conditions 
\begin{itemize} 
 \item $\Lambda_{\m_1}$ maps the space $\langle K^1_i~|~i=9..14\rangle$ into the 
space $\langle \vrho(A_i)~|~i=1..10,21\rangle$. This part of $\Lambda_{\m_1}$ 
depends on $66$ parameters.
 \item $\Lambda_{\m_1}$ maps the space $\langle K^1_i~|~i=1..8\rangle$ into the 
space $\langle \vrho(A_i)~|~i=11..18\rangle$. The corresponding $(8\times 8)$ 
matrix depends on $32$ parameters $a_i$, $i=1..32$, via the formulas
\bdm
 \begin{bmatrix}M^{1,2}&M^{3,4}&M^{5,6}&M^{7,8}&\\M^{9,{10}}
&M^{{11},{12}}&M^{{13},{14}}&M^{{15},{16}}\\
M^{{17},{18}}&M^{{19},{20}}&M^{{21},{22}}&M^{{23},{24}}\\
M^{{25},{26}}&M^{{27},{28}}&M^{{29},{30}}&M^{{31},{32}}\end{bmatrix},
\quad
M^{i,j}:=\begin{bmatrix}a_i&-a_j\\a_j&a_i\end{bmatrix}.
\edm
\end{itemize}
Since the torsion of the connection defined by $\Lambda_\m$ is given by
\cite[X.2.3]{KN1}
\be\label{eq:torsion}
T(X,Y)_o = \Lambda_\m(X)Y-\Lambda_\m(Y)X- [X,Y]_\m, \quad X,\,Y\in\m
\ee
we can calculate that the torsion
$T^{\alpha,..,\gamma}\in\Lambda^3(\SU(4)/\SO(2))$ 
if and only if $\alpha=\alpha_2=~..~=\alpha_8$ and
\bdm
\Lambda_{m_1}(K^1_{13})=\frac {\sqrt {2} 
\left( \alpha-\beta \right) }{\alpha\sqrt {\beta}}\vrho(A_9),~~
\Lambda_{m_1}(K^1_{14})
=\frac {\sqrt {2} \left( 
\alpha-\gamma \right) }{\alpha\sqrt {\gamma}}\vrho(A_{10}),
\edm
as well as $\Lambda_{m_1}(K^1_i)=0 \mbox{ for } i\neq 13,14$.
A closer look at the torsion shows that it never vanishes.
For the invariant torsion tensor and $X,Y,V\in\m$ we have
\be\label{eq:abltor}
(\nabla_VT)_o(X,Y)=\Lambda(V)T(X,Y)_o-T(\Lambda(V)X,Y)_o-T(X,\Lambda(V)Y)_o
\ee
and derive that $\nabla T^{\alpha\beta\gamma}=0$ for all $\alpha,\beta,\gamma>0$.
For $\gamma_{i,j,k}\in\Lambda^3(V^{14})$ and $s=1..21$ the standard 
representation $\nu$ of $Sp(3)$ on $\Lambda^3(V^{14})$ is given by:
\bdm
\nu(A_s)(\gamma_{i,j,k})
=\sum_{l}{\left(\gamma_{l,j,k}\cdot\vrho(A_s)_{l,i}
+\gamma_{i,j,l\cdot}\vrho(A_s)_{l,j}
+\gamma_{i,j,l}\cdot\vrho(A_s)_{l,k}\right)}.
\edm
We calculate the corresponding Casimir operator $C=\sum_{i=1}^{21}\nu(A_i)^2$ 
of this representation, which commutes with $\nu(A_i)$ for $i=1..21$. 
Therefore $C$ is given as a multiple of the identity on the irreducible 
components $\sp(3)$, $V^{70}$, $V^{84}$ and $V^{189}$. Its eigenvalues 
are $-8$, $-12$, $-18$ and $-16$. Applying the operator $C$ to the 
torsion, for any eigenvalue we obtain a  system of equations without solutions.

As stated in Corollary 4.2, Chapter 10 of \cite{KN2}, the Lie algebra of 
the holonomy group is given by 
\be\label{eq:hol}
\widetilde{\m_1}+[\Lambda_{\m_1}(\m_1),\widetilde{\m_1}]
+[\Lambda_{m_1}(\m_1),[\Lambda_{\m_1}(\m_1),\widetilde{\m_1}]]+\ldots 
\ee
where $\widetilde{\m_1}$ is spanned by all elements
\be\label{eq:mnull}
[\Lambda_{\m_1}(X),\Lambda_{\m_1}(Y)]-\Lambda_{\m_1}(proj_{\m_1}([X,Y]))-\vrho_1([X,Y]).
\ee
for $X,Y\in\m_1$.
With $T^3=\langle \vrho(A_9),\vrho(A_{10}),\vrho(A_{21})\rangle$ being the 
maximal torus in $\sp(3)\subset\so(\m)\cong\so(\m_1)$ we have 
$\Lambda_{\m_1}(\m_1)=\langle (\alpha-\beta)\vrho(A_9),(\alpha-\gamma)
\vrho(A_{10})\rangle\subset T^3$. Thus the first term in (\ref{eq:mnull}) 
vanishes and with $\vrho_1(\m_1)=\langle \vrho(A_{21})\rangle$ one easily gets
\bdm
\widetilde{\m_1}=\langle \vrho(A_{21}),(\alpha-\beta)\vrho(A_{9}),
(\alpha-\gamma)\vrho(A_{10})\rangle. 
\edm
With $\Lambda_{\m_1}(\m_1)\subset T^3$ and $\widetilde{\m_1}\subset T^3$ 
all except the first term of (\ref{eq:hol}) vanish and we get the algebra 
of the holonomy group equal to $\widetilde{\m_1}$.
\end{proof}
\begin{lem}[Curvature properties]
For any characteristic connection $\nabla^{\alpha\beta\gamma}$, the 
Ricci tensor in the constructed basis is given by
($a:=2\alpha-\gamma$, $b:=2\alpha-\beta$, $c:=2\alpha-\beta-\gamma$) 
\bdm
\Ric^{\nabla^{\alpha\beta\gamma}}\ =\ 
\frac{1}{\alpha^2}\ \diag(a,a,a,a,b,b,b,b,c,c,c,c,0,0).
\edm
Thus the scalar curvature is given by
\bdm
\Scal^{\nabla^{\alpha\beta\gamma}}=\frac{8(3\,\alpha-\beta-\gamma)}{\alpha^2}.
\edm
The Riemannian Ricci tensor  is for  
$a:=6\alpha-\gamma$, $b:=6\alpha-\beta$ and $c:=6\alpha-\beta-\gamma$ given by
\bdm
\Ric^g\ =\ \frac{1}{2\alpha^2} \diag(a,a,a,a,b,b,b,b,c,c,c,c,4\beta,4\gamma)
\edm
with scalar curvature
\bdm
\Scalg \ =\ \frac{2(18\,\alpha-\beta-\gamma)}{\alpha^2}.
\edm
In particular, such a manifold is never $\nabla^{\alpha\beta\gamma}$-Einstein 
nor Einstein in the Riemannian sense.
\end{lem}
\begin{proof}
We calculate the Ricci tensor $\Ric^{\nabla^{\alpha\beta\gamma}}$ for the 
characteristic connection in the constructed basis. Since 
$\Ric^{\nabla^{\alpha\beta\gamma}}$ is symmetric, with \cite{IP00} we get the identity
\be\label{eq:ric}
\Ric^g(X,Y)\ =\ \Ric^{\nabla^{\alpha\beta\gamma}}(X,Y)+
\frac{1}{4}\sum_{i=1}^{14}g^{\alpha\beta\gamma}(T^{\alpha\beta\gamma}(X,K_i^1),
T^{\alpha\beta\gamma}(Y,K_i^1))
\ee
and calculate the Ricci tensor for the Levi Civita connection $\Ric^g$.
\end{proof}
%
%
%
\subsection{The homogeneous space \boldmath$\U(4)/\SO(2)\times\SO(2)$}
\label{ch:ex2}$ $
\smallskip\\
We parametrize $H_2:=\SO(2)\times \SO(2)$ by a pair of
real numbers $(t_1,t_2)$
\bdm
\left( \begin{bmatrix} \cos t_1 & \sin t_1 \\ -\sin t_1&\cos t_1
  \end{bmatrix}, \begin{bmatrix} \cos t_2 & \sin t_2 \\ 
-\sin t_2 &\cos t_2 \end{bmatrix} \right)\in\SO(2)\times \SO(2)=: H_2
\edm
and embed  the Lie group $H_2$ in
$\U(4)=:K_2$ by: 
\begin{align*}
H_2\ra K_2, ~ (t_1,t_2) &\mapsto \diag\left(e^{\frac{i}{2}(t_1-t_2)},
e^{\frac{i}{2}(t_1+t_2)},e^{\frac{i}{2}(-t_1+t_2)},e^{\frac{i}{2}(-t_1-t_2)}\right).
\end{align*}
The action of $\h_2=\so(2)\oplus\so(2)$ splits the irreducible 
$14$-dimensional $\Sp(3)$-representation $V^{14}$ in six $2$-dimensional 
representations and two trivial ones. We choose an invariant metric
\bdm 
g^{\alpha,..,\gamma}=\diag(\alpha,\alpha,\alpha_{2},\alpha_{2},\alpha_{3},
\alpha_{3},\alpha_{4},\alpha_{4},\alpha_{5},\alpha_{5},\alpha_{6},\alpha_{6},\beta,\gamma)
\edm
with an orthonormal basis $K^2_1,\ldots, K^2_{14}$ of $\m_1$ as done in  
Appendix \ref{ch:ap:u4}. 
\begin{thm}
Consider the manifold $M_2=\U(4)/\SO(2)\times\SO(2)$ equipped with the
metric $g^{\alpha,..,\gamma}$.
For general parameters $\alpha,\alpha_i,\beta,\gamma>0$, it carries a 
$30$-dimensional space of invariant $Sp(3)$ connections,
and for $\alpha=\alpha_2=..=\alpha_6$, the $Sp(3)$ structure
admits a characteristic connection with torsion 
$T^{\alpha\beta\gamma}\in\Lambda^3(M_2)$. These $Sp(3)$ structures
with characteristic connection have the following properties:
\begin{enumerate}
\item The  characteristic connection has alwas parallel torsion,
$\nabla^{\alpha\beta\gamma}T^{\alpha\beta\gamma}=0$.
\item The structure is never integrable.
\item The structure is of mixed type.
\item The Lie algebra of the holonomy group of the characteristic 
connection is a subalgebra of the maximal torus of $\sp(3)$ and it is
\begin{itemize}
\item two-dimensional, if $\alpha\neq\gamma$ and
\item three-dimensional, if $\alpha\neq\gamma$.
\end{itemize}
\end{enumerate}
\end{thm}
\begin{proof}
We calculate all invariant $\Sp(3)$-connections via their corresponding 
equivariant maps $\Lambda_{\m_2}:\m_2\ra\sp(3)$ and get all connections 
via maps $\Lambda_{\m_2}$ satisfying the following $5$ conditions 
with parameters $a_i$, $i=1,\ldots,20$
\begin{itemize}
\item $\Lambda_{\m_2}$ maps the space $\langle K^2_i~|~i=1,2\rangle$ into the 
space $\langle \vrho(A_i)~|~i=11,12\rangle$ and the corresponding matrix has the form
\bdm
\Lambda_{\m_2}|_{\langle K^2_i~|~i=1,2\rangle}=\begin{bmatrix}a_1&-a_2\\a_2&a_1\end{bmatrix},
\edm
\item $\Lambda_{\m_2}$ maps the space $\langle K^2_i~|~i=3,4\rangle$ into the 
space $\langle \vrho(A_i)~|~i=13,14\rangle$ and the corresponding matrix has the form
\bdm
\Lambda_{\m_2}|_{\langle K^2_i~|~i=3,4\rangle}=\begin{bmatrix}a_3&-a_4\\a_4&a_3\end{bmatrix},
\edm
\item $\Lambda_{\m_2}$ maps the space $\langle K^2_i~|~i=5..8\rangle$ into the 
space $\langle \vrho(A_i)~|~i=15..18\rangle$ and the corresponding matrix has the form
\bdm
\Lambda_{\m_2}|_{\langle K^2_i~|~i=5..8\rangle}\ =\ 
\begin{bmatrix}a_5&-a_6&a_7&-a_8\\a_6&a_5&a_8&a_7\\
a_9&-a_{10}&a_{11}&-a_{12}\\a_{10}&a_{9}&a_{11}&a_{12}\end{bmatrix},
\edm
\item $\Lambda_{\m_2}$ maps the space $\langle K^2_i~|~i=9..12\rangle$ into the 
space $\langle \vrho(A_i)~|~i=1..4\rangle$ and the corresponding matrix has the form
\bdm
\Lambda_{\m_2}|_{\langle K^2_i~|~i=9..12\rangle}\ = \ 
\begin{bmatrix}a_{13}&-a_{14}&-a_{15}&a_{16}\\a_{14}&a_{13}&a_{16}&a_{15}\\
-a_{17}&a_{18}&a_{19}&-a_{20}\\a_{18}&a_{17}&a_{20}&a_{19}\end{bmatrix}
\edm
\item $\Lambda_{\m_2}$ maps the space $\langle K^2_i~|~i=13,14\rangle$ into the 
space $\langle \vrho(A_i)~|~i=5,7,9,10,21\rangle$. This part depends on $10$ 
parameters, other than $a_i$, $i=1..20$.
\end{itemize}
With  equation (\ref{eq:torsion}) we compute the torsion tensor, 
which is skew symmetric if and only if $\alpha=\alpha_2=\ldots =\alpha_6$ and
\bdm
\Lambda_{\m_2}(K^2_{13})
=\frac{\sqrt{2}(\alpha-\beta)}{\alpha\sqrt{\beta}}\vrho(A_9) \mbox{ and } 
\Lambda_{\m_2}(K^2_i)=0 \mbox{ for } i \neq 13.
\edm
For such connections $\nabla^{\alpha\beta\gamma}$ the torsion never 
vanishes. Again we compute that the torsion is parallel for all 
such connections and that none of the torsion tensors lies in 
any eigenspace of the Casimir operator.\\
 
With the formulas (\ref{eq:hol}) and (\ref{eq:mnull}) we get for 
the maximal torus  $T^3$ in $\sp(3)$ that 
$\Lambda_{\m_2}(\m_2)=\langle (\alpha-\beta)\vrho(A_9)\rangle\subset T^3$. 
Thus the first term in (\ref{eq:mnull}) again vanishes and with 
$\vrho_2(\m_2)=\langle \vrho(A_{10}),\vrho(A_{21})\rangle$ one easily gets
\bdm
\widetilde{\m_2}=\langle \vrho(A_{21}),(\alpha-\beta)\vrho(A_{9}),\vrho(B_{10})\rangle
\edm
and again we get the Lie algebra of the holonomy group being $\widetilde{\m_2}$.
\end{proof}
\begin{lem}[Curvature properties]\label{lem:curv2} 
On $M_2$, the Ricci tensor for the characteristic connection is 
 given by  ($a:=2\alpha-\beta$)
\bdm 
\Ric^{\nabla^{\alpha\beta\gamma}}=\frac{1}{\alpha^2}
\diag(2\alpha,2\alpha,2\alpha,2\alpha,a,a,a,a,a,a,a,a,0,0) 
\edm 
with scalar curvature 
\bdm
\Scal^{\nabla^{\alpha\beta\gamma}}=\frac{8(3\,\alpha-\beta)}{\alpha^2}.
\edm
The Riemannian Ricci tensor for the Levi Civita is for $a:=6\alpha-\beta$
given by 
\bdm
\Ric^g=\frac{1}{2\alpha^2}\diag(6\alpha,6\alpha,6\alpha,6\alpha,a,a,a,a,a,a,a,a,4\beta,0) 
\edm
and  
\bdm
\Scal^g=\frac{2(18\,\alpha-\beta)}{\alpha^2}.
\edm
Thus this space is never $\nabla^{\alpha\beta\gamma}$-Einstein 
nor Einstein for the Levi Civita connection. 
\end{lem}
\begin{proof}
The proof follows immediately from the identity (\ref{eq:ric}).
\end{proof}
We will now have a look at invariant spinors on $M_2$.
 Since $M_2$ carries 
a unique homogeneous spin structure (see Remark \ref{nb:spin}), we can lift 
the characteristic connection $\nabla^{\alpha\beta\gamma}$ to the 
spin bundle. To use the map $\Lambda_\m$ for  calculations, 
we look at elements $\psi\in\Delta_{14}$ that are invariant under the lifted  
action of $\SO(2)\times\SO(2)$ defining global spinors via the constant  
map $\U(4)\ra\Delta_{14},~g\mapsto\psi$. We get a $16$-dimensional space  
of such invariant spinors.\\
The Dirac operator we will look at is the Dirac operator $\D$ of the   
connection  with torsion $T^{\alpha\beta\gamma}/3$. 
With the lifted map 
$\widetilde{\Lambda_\m}$ we easily compute for an invariant spinor $\psi$ 
\be\label{eq:diracop}
\D\psi=\sum_{i=1}^{14}\widetilde{\Lambda_\m}(K^2_i)\psi 
-\frac{1}{2}T^{\alpha\beta\gamma}\cdot\psi,
\ee
where the torsion $T^{\alpha\beta\gamma}$ is considered as a $3$-form  
and acts on a spinor via  Clifford multiplication. 
Since the dimension $14$ is even, the spinor bundle splits in two bundles  
being invariant under the $\Spin(n)$ action and we calculate
\begin{lem}\label{ev-Dirac}
 The lift of the action of $\SO(2)\times\SO(2)$ on the $128$-dimensional 
space $\Delta_{14}$ admits a $16$-dimensional space of invariant spinors.  
The Dirac operator $\D$ has the two eigenvalues 
$\pm\sqrt{\frac{\alpha+4\beta}{\alpha\beta}}$ on this space. 
\end{lem}
As  it is just a scaling of the metric, we can fix one parameter of the  
metric and hence choose $\alpha=1$. We look at the estimates for the 
first eigenvalue $\lambda$ of the Dirac operator   
valid  for a connection with parallel  torsion.
From the results of \cite{Agricola&F04}, it follows that
\be\label{eq:dirac1}
\lambda^2\geq
\frac{1}{4}\Scal^g+\frac{1}{8}||T^{\alpha\beta\gamma}||^2-\frac{1}{4}\mu^2 
\ee
whereas the twistorial eigenvalue estimate derived in \cite{abk} states that
\be\label{eq:dirac2}
\lambda^2\geq \frac{14}{4(14-1)}\Scal^g 
+\frac{14(14-5)}{8(14-3)^2}||T^{\alpha\beta\gamma}||^2
+\frac{14(4-14)}{4(14-3)^2}\mu^2,
\ee
where $\mu$ is the largest eigenvalue of the operator $T^{\alpha\beta\gamma}$.
Typically, it  depends on the underlying geometry which
of the inequalities is better (see \cite{abk} for a detailed discussion). 
We calculate the operator $T^{\alpha\beta\gamma}$ for the orthonormal  
basis $K_i^2$, $i=1,\ldots,14$ of $\m$ for any $v\in\m$ as
\bdm 
T^{\alpha\beta\gamma}v
=\sum_{i,j,k=1}^{14}T^{\alpha\beta\gamma}(K^2_i,K^2_j,K^2_k)K^2_i\cdot
K^2_j\cdot K^2_k\cdot v.
\edm 
This yields the eigenvalues $\mu=\pm2\sqrt{4+\beta}$ of  
$T^{\alpha\beta\gamma}$ on the space of invariant spinors and with 
\bdm
||T^{\alpha\beta\gamma}||^2
=\sum_{i,j,k=1}^{14}T^{\alpha\beta\gamma}(K^2_i,K^2_j,K^2_k)^2=8+4\beta
\edm
and Lemma \ref{lem:curv2} we obtain:
The estimate (\ref{eq:dirac1}) is equal to the square of the eigenvalue
computed in Lemma \ref{ev-Dirac} if $\beta=1$, and indeed in this  case 
one checks that all invariant spinors are parallel.
The estimate (\ref{eq:dirac2}) 
is always strict,  hence there does not exist a metric for which
an invariant spinor becomes a twistor spinor with torsion.
The twistorial  estimate is stronger then the first one if
$\beta<\frac{166}{275}$.\\
The manifold $M_1$ considered in Section \ref{ex:1} carries a 
$48$-dimensional spaces of 
invariant spinors and the computer was not able to compute the 
eigenvalues of the corresponding Dirac operator.
%
%
%
%
\subsection{The homogeneous space
\boldmath$\U(4)\times\U(1)/\SO(2)\times\SO(2)\times\SO(2)$}$ $ 
\smallskip\\
In this example, we shall parametrize the Lie group 
$H_3:=\SO(2)\times \SO(2)\times \SO(2)$ as
\bdm
\left( \begin{bmatrix} \cos t_1 & \sin t_1 \\ -\sin t_1& \cos t_1 \end{bmatrix},
\begin{bmatrix} \cos t_2 & \sin t_2 \\ -\sin t_2&\cos t_2 \end{bmatrix},
\begin{bmatrix} \cos t_3 & \sin t_3 \\ -\sin t_3&\cos t_3 \end{bmatrix} 
\right)\in \SO(2)\times \SO(2)\times \SO(2),
\edm
and embed it  into $K_3=\U(4)\times \U(1)$ by
\bdm
(t_1,t_2,t_3) \mapsto 
\left( \diag\left(e^{\frac{i}{2}(t_1+t_2-t_3)},e^{\frac{i}{2}(t_1-t_2+t_3)},
e^{\frac{i}{2}(-t_1+t_2+t_3)},e^{\frac{i}{2}(-t_1-t_2-t_3)} \right),1\right).
\edm
The action of $\h_3=\so(2)\oplus\so(2)\oplus\so(2)$ splits $V^{14}$ in the 
same irreducible representations as the representation of 
$\h_2=\so(2)\oplus\so(2)$, so we can choose the same Ansatz for the metric
\bdm g^{\alpha,..,\gamma}=\diag(\alpha,\alpha,\alpha_{2},\alpha_{2},\alpha_{3},
\alpha_{3},\alpha_{4},\alpha_{4},\alpha_{5},\alpha_{5},\alpha_{6},\alpha_{6},\beta,\gamma)
\edm
with an orthonormal basis $\{K^3_i~|~i=1..14\}$ of $\m_3$ as described in 
Appendix \ref{ch:ap:u4u1}.
\begin{thm}
Consider the manifold $M_3=\U(4)\times\U(1)/\SO(2)\times\SO(2)\times\SO(2)$ 
equipped with the metric $g^{\alpha,..,\gamma}$.
For any parameters $\alpha,\alpha_i,\beta,\gamma>0$, it carries
an  $18$-dimensional space of invariant $Sp(3)$-connections,
and for $\alpha=\alpha_2=..=\alpha_6$, the $Sp(3)$ structure
admits a characteristic connection with torsion 
$T^{\alpha\beta\gamma}\in\Lambda^3(M_3)$. These $Sp(3)$ structures
with characteristic connection have the following properties:
\begin{enumerate}
\item The  characteristic connection has alwas parallel torsion,
 $\nabla^{\alpha\beta\gamma}T^{\alpha\beta\gamma}=0$, and it coincides with
the canonical connection.
\item The structure is never integrable.
\item The structure is of mixed type.
\item The Lie algebra of the holonomy group of the characteristic 
connection is the maximal torus in $\sp(3)$. 
\end{enumerate}
\end{thm}
\begin{proof}
We get all possible $\Sp(3)$ connections via equivariant maps 
$\Lambda_{\m_3}:\m_3\ra\sp(3)$ with parameters $a_i$, $i=1..12$ and 
the conditions
\begin{itemize}
 \item for pairs $(i,j) \in \{(1,11),(3,13),(5,17),(7,15),(9,1),(11,3)\}$ we have
\bdm
\Lambda_{\m_3}(K^3_i)=a_i\vrho(A_j)+a_{i+1}\vrho(A_{j+1}) \mbox{ and } 
\Lambda_{\m_3}(K^3_i)=-a_{i+1}\vrho(A_j)+a_i\vrho(A_{j+1}),
\edm
\item $\Lambda_{\m_3}$ maps the space $\langle K^3_i~|~i=13,14\rangle$ into 
the space $\langle \vrho(A_i)~|~i=9,10,21\rangle$, which is dependent on $6$ parameters.
\end{itemize}
We again get a skew symmetric torsion if and only if
 $\alpha=\alpha_2=..=\alpha_6$ 
with the only possible invariant $\Sp(3)$ connection being the 
canonical connection defined by $\Lambda_{\m_3}\equiv 0$. 
For this connections $\nabla^{\alpha\beta\gamma}$ the torsion never 
vanishes. Again we compute that the torsion is parallel for all 
such connections and that none of the torsion tensors lie in any 
eigenspace of the Casimir operator.\\

Since $\Lambda_{m_3}\equiv 0$ we get the Lie algebra of the holonomy 
group via the formulas (\ref{eq:hol}) and (\ref{eq:mnull}) being equal to
\bdm
\vrho_3(proj_{\h_3})([\m_3,\m_3])\ =\ \vrho_3(\h_3)
\edm
and thus being the maximal torus in $\sp(3)$ (see Appendix \ref{ch:ap:u4u1}).
\end{proof}
\begin{lem}[Curvature properties]
The Ricci tensor for the characteristic connection is given by
\bdm
\Ric^{\nabla^{\alpha\beta\gamma}}=\frac{2}{\alpha}\diag(1,1,1,1,1,1,1,1,1,1,1,1,0,0)
\edm
and its scalar curvature is $\Scal^{\nabla^{\alpha\beta\gamma}}
=\frac{24}{\alpha}$.\\
The Riemannian Ricci tensor for the Levi Civita is given 
by $\Ric^g=\frac{3}{2}\, \Ric^{\nabla^{\alpha\beta\gamma}}$. Thus the Riemannian 
scalar curvature is $\Scal^g=\frac{36}{\alpha}$, and the space is never 
$\nabla^{\alpha\beta\gamma}$-Einstein nor  Riemannian Einstein.
\end{lem}
The action of $\SO(2)\times\SO(2)\times\SO(2)$ lifted in $\Delta_{14}$ 
has no trivial parts and thus there are no invariant spinors. Hence
it is not possible to make any statements on the spectrum of the 
Dirac operator.
%
%
\begin{NB}
In the first $3$ examples, we constructed $\Sp(3)$ spaces via embeddings of
$K_i$ 
in the maximal torus
$T^3$ of $\SU(4)$. Since $\vrho_3(T^3)\subset \Sp(3)$,
 one can choose any embedding 
of $K_i$ into $T^3$ to get a $\Sp(3)$ manifold $K_i/H_i$.
For those examples there are different possible identifications 
$\m_i\cong\m$  giving different identifications 
$\SO(\m_i)\cong\SO(\m)\supset\Sp(3)$, 
such that $\rho_i(SO(2)^i)\subset\Sp(3)$. 
Those induce different $\Sp(3)$ 
structures on the given manifolds, but their geometry is just the same.
\end{NB}
%
\subsection{ The homogeneous space 
\boldmath$\SU(5)/\Sp(2)$}\label{ex:su5}$ $
\smallskip\\
We restrict $A_i$, $i=1..10$ to the lower $5\times 5$-matrix and get the Lie 
algebra of $H_4=\Sp(2)$ in $\k_4=\su(5)$. In \cite{ker}, it was shown that
with this embedding $\SU(5)\subset \SU(6)$, the Lie group $K_4=\SU(5)$ 
already acts transitively on $\SU(6)/\Sp(3)$ with isotropy group 
$H_4=\Sp(2)$. As a manifold, $\SU(6)/\Sp(3)$ is hence diffeomorphic 
to $\SU(5)/\Sp(2)$, but the homogeneous structure is a different one. The 
adjoint representation $\vrho_4$ of $\Sp(2)=H_4$ on this space is just a 
restriction of the action of $\Sp(2)\subset\Sp(3)$ on $\m_4$ 
(see Appendix \ref{ch:ApSU5}) and we get an $\Sp(3)$ structure on 
$\SU(5)/\Sp(2)$. The representation $\vrho_4$ splits 
$\m_4=\Delta_5\oplus\p^5\oplus\p^1$ as 
shown in Theorem \ref{thm:maxsub} and we get an  $3$-dimensional family 
of invariant metrics using multiples $\alpha,\beta,\gamma>0$ of the 
negative of the Killing form on each component,
\bdm
g^{\alpha\beta\gamma}=\diag(\alpha,\alpha,\alpha,\alpha,\alpha,\alpha,
\alpha,\alpha,\beta,\beta,\beta,\beta,\beta,\gamma).
\edm
\begin{thm}\label{th:ex4}
Consider the manifold $M_4=\SU(5)/\Sp(2)$ equipped with  the metric
$g^{\alpha\beta\gamma}$.
For any parameters $\alpha,\beta,\gamma>0$,  it carries
an  $7$-dimensional space of invariant $Sp(3)$ connections,
and to each of these metrics corresponds exactly one
characteristic connection with torsion
$T^{\alpha\beta\gamma}\in\Lambda^3(V^{14})$.
These $Sp(3)$ structures
with characteristic connection have the following properties:
\begin{enumerate}
\item The characteristic connection satisfies $\nabla^{\alpha\beta\gamma}
T^{\alpha\beta\gamma}=0$ if and only if either 
\begin{itemize}
\item $\beta=\alpha$ or 
\item $\beta=2\alpha$ and $\gamma=\frac{6}{5}\alpha.$
\end{itemize}
\item The structure is
\begin{itemize}
\item integrable  if $\beta=2\alpha$ and $\gamma=\frac{6}{5}\alpha$,
\item of type $\sp(3)$ if $\alpha=\frac{1}{4}(\sqrt{15\beta\gamma}-\beta)$,
\item of type $V^{189}$ if $\alpha=\frac{1}{12}(9\beta-\sqrt{15\beta\gamma})$.
\end{itemize}
\item The Lie algebra of  
the holonomy group of the characteristic connection is given by
\begin{itemize}
\item $\sp(3)$ if $\alpha\neq\beta$,
\item $\sp(2)\oplus W^1$ if $\gamma\neq\alpha=\beta$, where $W^1$ is the 
one-dimensional subspace in the maximal torus $T^3$ of $\sp(3)$ such 
that $T^3\subset\sp(2)\oplus W^1$ and
\item $\sp(2)$ if $\alpha=\beta=\gamma$.
\end{itemize}
\end{enumerate}
\end{thm}
\begin{NB}
 In case of an integrable structure, $T^{\alpha\beta\gamma}=0$, 
$\SU(5)/\Sp(2)$ locally isometric to a symmetric space, as mentioned 
before \cite{nur}.
\end{NB}
\begin{proof}
Again we look at linear maps
$\Lambda_{\m_4}: \, {\m_4}\ra\sp(3)$ that are equivariant 
under the representation $\vrho_4$.
\bdm
\Lambda_{\m_4} (\vrho_4(h)X)\ =\ \vrho_4(h)\Lambda_{\m_4} (X)\vrho_4(h)^{-1} \quad
\forall h\in Sp(2),\ X\in \m_4.
\edm
One calculates that this is the case if and only if
 $\Lambda_{\m_4}$ fulfills the following conditions
\begin{itemize}
 \item $\Lambda_{\m_4}$ is identically zero on $\p^5$.
 \item $\Lambda_{\m_4}$ maps $p^1$ into the space
   $\langle \vrho(A_i)~|~i=19..21\rangle$. 
This gives $3$ parameters.
 \item $\Lambda_{\m_4}$ maps $\Delta_5$ into the space 
$\langle \vrho(A_i)~| ~i=11..18\rangle$ and the corresponding matrix is for 
$a,b,c,d\in \R$ given by 
\bdm
\Lambda_{\m_4}|_{\Delta_5}=\begin{bmatrix}b&-a&-d&c&&&&\\a&b&c&d&&&&\\
d&-c&b&-a&&&&\\-c&-d&a&b&&&&\\&&&&d&-c&b&-a\\&&&&-c&-d&a&b\\&&&&b&-a&-d&c\\&&&&a&b&c&d 
\end{bmatrix}.
\edm
\end{itemize} 
Since $\sp(2)=\langle \vrho(A_i)~|~i=1..10\rangle$ and $Im(\Lambda_{\m_4})\cap 
\langle \vrho(A_i)~|~i=1..10\rangle=\{0\}$, the only $\sp(2)$-connection is 
the canonical connection.\\ 
With (\ref{eq:torsion}), we calculate the torsion. The condition 
$T^{\alpha\beta\gamma}\in\Lambda^3(\SU(5)/\Sp(2))$ for the torsion tensor 
implies $a=c=d=0,\ b=\frac{\alpha-\beta}{\alpha\sqrt{\beta}}$ and 
\bdm
\Lambda_{\m_4}|_{\p^1}:\p^1\ra \langle \vrho(A_{21})\rangle
\edm
is given by the multiplication with the constant  
$\frac{1}{\sqrt{2}}{\frac { -\gamma\sqrt{5\beta}+ \sqrt{3\gamma}\beta 
-\sqrt{3\gamma}\alpha+\sqrt{5\beta}\alpha}{\alpha\sqrt{\beta\gamma}}}$.\\
Again with  equation (\ref{eq:abltor}) we derive that 
$\nabla^{\alpha\beta\gamma}T^{\alpha\beta\gamma}=0$ if and only if either  
$(\beta=\alpha)$ or $(\beta=2\alpha$ and $\gamma=\frac{6}{5}\alpha)$.
One computes that $T^{\alpha\beta\gamma}=0$ iff $\beta=2\alpha$ and  
$\gamma=\frac{6}{5}\alpha$.
With the above calculated Casimir operator $C$ we get 
\bdm
C(T^{\alpha\beta\gamma})=-8T^{\alpha\beta\gamma} \Leftrightarrow \alpha 
=\frac{1}{4}(\sqrt{15\beta\gamma}-\beta)
\edm
and 
\bdm
C(T^{\alpha\beta\gamma})=-16T^{\alpha\beta\gamma} \Leftrightarrow \alpha 
=\frac{1}{12}(9\beta-\sqrt{15\beta\gamma}).
\edm
With equations (\ref{eq:hol}), (\ref{eq:mnull}) and an  
appropriate computer algebra program one computes that the Lie algebra 
of the holonomy group is given by $\sp(3)$ if $\alpha\neq\beta$ and by  
$\sp(2)\oplus \langle (\alpha-\gamma)B_{21}\rangle$ if $\alpha=\beta$.  
\end{proof}
We calculate the Ricci tensor for the  
characteristic connection  and the Levi Civita connection from
equation (\ref{eq:ric}).
\begin{lem}[Curvature properties] 
The Ricci tensor for the characteristic connection is for  
$a:=\frac{2\sqrt{15\beta\gamma}-11\beta-5\gamma}{4\alpha^2}+\frac{21}{2\alpha}
-\frac{4}{\beta}-\frac{\sqrt{15\gamma}}{2\alpha\sqrt{\beta}}$,  
$b:=\frac{2(\alpha+\beta)}{\beta\alpha}$,  
$c:=2(\beta-\alpha)(\frac{3}{\alpha\beta}
+\frac{\sqrt{15\gamma}}{\alpha^2\sqrt{\beta}}-\frac{3}{\alpha^2})$ 
given by
\bdm
\Ric^{\nabla^{\alpha\beta\gamma}}=\diag(a,a,a,a,a,a,a,a,b,b,b,b,b,c).
\edm
The Riemannian Ricci tensor is for 
$a:=10\alpha-\frac{5}{4}\beta-\frac{5}{4}\gamma$ and 
$b=\frac{8\alpha^2+\beta^2}{\beta}$ equal to
\bdm
 \Ric^g=\frac{1}{2\alpha^2}\diag(a,a,a,a,a,a,a,a,b,b,b,b,b,5\gamma).
\edm
Its scalar curvature is  
$\Scal^g=\frac{5(16\alpha\beta-\beta\gamma-\beta^2+8\alpha^2)}{2\alpha^2\beta}$.
Thus, this space is a Riemannian Einstein space if 
$\sqrt{2}\alpha=\beta=\frac{1}{\sqrt{8}-1}\gamma$ and in this case we have
\bdm
\Ric^g=\frac{5}{2\alpha^2}g^{\alpha\beta\gamma}. 
\edm
%
\end{lem}
We lift the representation of $\Sp(2)$ in $\SO(14)$ to $\Spin(14)$ 
and with the formula  (\ref{eq:diracop}) we calculate
\begin{lem}
$\Delta_{14}$ has a $4$-dimensional space of $\Sp(2)$ invariant spinors 
and the Dirac operator $\D$ has eigenvalues  
\bdm
\pm\frac{1}{2}\sqrt{\frac{5\alpha^2\beta +3\alpha^2\gamma  
-6\alpha\beta\gamma +2\alpha\sqrt{15\beta\gamma}(\beta -\alpha)  
+28\beta^2\gamma}{\alpha^2\beta\gamma}}.
\edm
\end{lem}
As in Section \ref{ch:ex2}, we restrict the general case, ignoring the 
possible scaling, to the case $\alpha=1$. To look at the inequalities  
(\ref{eq:dirac1}) and (\ref{eq:dirac2}) we need the torsion to be parallel.
From the two possible cases mentioned in  Theorem \ref{th:ex4}, only
the first is of interest, since the torsion vanishes in the second
and Friedrich's Riemannian estimate from 1980 applies.

So, assume that $\beta=\alpha=1$.
 The operator $T^{\alpha\beta\gamma}$ has eigenvalues 
 $\mu=\pm\sqrt{25+5\gamma}$ 
and its norm is given by $||T^{\alpha\beta\gamma}||^2=5+5\gamma$. Thus we
obtain that the estimate (\ref{eq:dirac2}) is always strict,
and the estimate (\ref{eq:dirac1}) becomes an equality for 
$\gamma=\beta=\alpha=1$.  
As expected, all invariant spinors  are parallel for $\gamma=\beta=\alpha=1$. 
The inequality (\ref{eq:dirac2})  
is better than the inequality (\ref{eq:dirac1}) if 
$\gamma<\frac{189}{275}$.

%
%
%
\appendix \section{Explicit realizations of representations \& other geometric
data}\label{ch:ap:sp3}
%
Let $\{e_i^n\}_{i=1..n}$ be the standard basis of $\R^n$, $E^n_{i,j}\in\su(n)$
the matrix given by the linear map $e_i^n\mapsto-e_j^n$, $e_j^n\mapsto e_i^n$ 
and $S^n_{i,j}$ given by $e_i^n\mapsto e_j^n$, $e_j^n\mapsto e_i^n$. We used throughout 
the following basis $A_1,\ldots,A_{21}$ of the Lie algebra of $\Sp(3)\subset \SU(6)$, 
\begin{gather*}
A_{1}   :=\frac{1}{2}(E^6_{2,3}+ E^6_{5,6}), \ \
A_{2}  :=\frac{i}{2}(S^6_{2,3}- S^6_{5,6}),  \ \
A_{3}  :=\frac{1}{2}(E^6_{2,6}+ E^6_{3,5}),\ 
A_{4}  :=\frac{i}{2}(S^6_{2,6}+ S^6_{3,5}),\\
A_{5}  :=\frac{1}{\sqrt{2}}E^6_{2,5},  \ \
A_{6}  :=\frac{1}{\sqrt{2}}E^6_{3,6},  \ \
A_{7}  :=\frac{i}{\sqrt{2}}S^6_{2,5},  \  \
A_{8}  :=\frac{i}{\sqrt{2}}S^6_{3,6},\ \
A_{9}  :=\frac{i}{\sqrt{2}}(S^6_{2,2}-S^6_{5,5}),  \\
A_{10}  :=\frac{i}{\sqrt{2}}(S^6_{3,3}-S^6_{6,6}),\  \
A_{11}  :=\frac{1}{2}(E^6_{1,3}+ E^6_{4,6}), \ \
A_{12}  :=\frac{i}{2}(S^6_{1,3}- S^6_{4,6}), \ \
A_{13}  :=\frac{1}{2}(E^6_{1,6}+ E^6_{3,4}),  \\ 
A_{14}  :=\frac{i}{2}(S^6_{1,6}+ S^6_{3,4}), \ \
A_{15}  :=\frac{1}{2}(E^6_{1,5}+ E^6_{2,4}), \ \
A_{16}  :=\frac{i}{2}(S^6_{1,5}+ S^6_{2,4}),\ \
A_{17}  :=\frac{1}{2}(E^6_{1,2}+ E^6_{4,5}),  \\
A_{18}  :=\frac{i}{2}(S^6_{1,2}- S^6_{4,5}), \ \
A_{19}  :=\frac{1}{\sqrt{2}}E^6_{1,4},\  \
A_{20}  :=\frac{i}{\sqrt{2}}S^6_{1,4},\ \ 
A_{21}  :=\frac{i}{\sqrt{2}}(S^6_{1,1}-S^6_{4,4}).            
\end{gather*}
Hence, we get a basis of $\m$, $\su(6)=\m\oplus\sp(3)$ as
\bdm
B_{1}:=\frac{1}{2}(E^6_{1,3}- E^6_{4,6}),~
B_{2}:=\frac{i}{2}(S^6_{1,3}+ S^6_{4,6}),~
B_{3}:=\frac{1}{2}(E^6_{1,6}- E^6_{3,4}),\edm\bdm
B_{4}:=\frac{i}{2}(S^6_{1,6}- S^6_{3,4}),~
B_{5}:=\frac{1}{2}(E^6_{1,2}-E^6_{4,5}),~
B_{6}:=\frac{i}{2}(S^6_{1,2}+S^6_{4,5}),\edm\bdm
B_{7}:=\frac{1}{2}(E^6_{1,5}-E^6_{2,4}),~
B_{8}:=\frac{i}{2}(S^6_{1,5}-S^6_{2,4}),~
B_{9}:=\frac{1}{2}(E^6_{2,3}-E^6_{5,6}),\edm\bdm
B_{10}:=\frac{i}{2}(S^6_{2,3}+S^6_{5,6}),~
B_{11}:=\frac{1}{2}(E^6_{2,6}-E^6_{3,5}),~
B_{12}:=\frac{i}{2}(S^6_{2,6}-S^6_{3,5}),\edm\bdm
B_{13}:=\frac{i}{2}(S^6_{2,2}-S^6_{3,3}+S^6_{5,5}-S^6_{6,6}),~
B_{14}:=\frac{i}{2\sqrt{3}}(-2S^6_{1,1}+S^6_{2,2}+S^6_{3,3}-2S^6_{4,4}+S^6_{5,5}+S^6_{6,6}).
\edm
The isotropy representation of $\sp(3)$ on $\m\cong V^{14}$  is thus
\bdm
\vrho(A_{1})=
-\frac{1}{2}E^{14}_{1,5}
-\frac{1}{2}E^{14}_{2,6}
-\frac{1}{2}E^{14}_{3,7}
-\frac{1}{2}E^{14}_{4,8},
-E^{14}_{10,13}, \quad
\vrho(A_{2})=
\frac{1}{2}E^{14}_{1,6}
-\frac{1}{2}E^{14}_{2,5}
-\frac{1}{2}E^{14}_{3,8}
+\frac{1}{2}E^{14}_{4,7}
+E^{14}_{9,13}
\edm\bdm
\vrho(A_{3})=
\frac{1}{2}E^{14}_{1,7}
+\frac{1}{2}E^{14}_{2,8}
-\frac{1}{2}E^{14}_{3,5}
-\frac{1}{2}E^{14}_{4,6}
-E^{14}_{12,13},\quad
\vrho(A_{4})=
\frac{1}{2}E^{14}_{1,8}
-\frac{1}{2}E^{14}_{2,7}
+\frac{1}{2}E^{14}_{3,6}
-\frac{1}{2}E^{14}_{4,5}
+E^{14}_{11,13}
\edm\bdm
\vrho(A_{5})=
\frac{\sqrt{3}}{2}E^{14}_{5,7}
+\frac{\sqrt{3}}{2}E^{14}_{6,8}
+\frac{\sqrt{3}}{2}E^{14}_{9,11}
-\frac{\sqrt{3}}{2}E^{14}_{10,12},\quad
\vrho(A_{6})=
\frac{\sqrt{3}}{2}E^{14}_{1,3}
+\frac{\sqrt{3}}{2}E^{14}_{2,4}
+\frac{\sqrt{3}}{2}E^{14}_{9,11}
+\frac{\sqrt{3}}{2}E^{14}_{10,12}
\edm\bdm
\vrho(A_{7})=
\frac{\sqrt{3}}{2}E^{14}_{5,8}
-\frac{\sqrt{3}}{2}E^{14}_{6,7}
+\frac{\sqrt{3}}{2}E^{14}_{9,12}
+\frac{\sqrt{3}}{2}E^{14}_{10,11},\quad
\vrho(A_{8})=
\frac{\sqrt{3}}{2}E^{14}_{1,4}
-\frac{\sqrt{3}}{2}E^{14}_{2,3}
+\frac{\sqrt{3}}{2}E^{14}_{9,12}
-\frac{\sqrt{3}}{2}E^{14}_{10,11}
\edm\bdm
\vrho(A_{9})=
\frac{\sqrt{3}}{2}E^{14}_{5,6}
-\frac{\sqrt{3}}{2}E^{14}_{7,8}
-\frac{\sqrt{3}}{2}E^{14}_{9,10}
-\frac{\sqrt{3}}{2}E^{14}_{11,12},\quad
\vrho(A_{10})=
\frac{\sqrt{3}}{2}E^{14}_{1,2}
-\frac{\sqrt{3}}{2}E^{14}_{3,4}
+\frac{\sqrt{3}}{2}E^{14}_{9,10}
-\frac{\sqrt{3}}{2}E^{14}_{11,12}
\edm\bdm
\vrho(A_{11})=
-\frac{1}{2}E^{14}_{2,13}
+\frac{\sqrt{3}}{2}E^{14}_{2,14}
-\frac{1}{2}E^{14}_{5,9}
+\frac{1}{2}E^{14}_{6,10}
-\frac{1}{2}E^{14}_{7,11}
-\frac{1}{2}E^{14}_{8,12}
\edm\bdm
\vrho(A_{12})=
\frac{1}{2}E^{14}_{1,13}
-\frac{\sqrt{3}}{2}E^{14}_{1,14}
-\frac{1}{2}E^{14}_{5,10}
-\frac{1}{2}E^{14}_{6,9}
+\frac{1}{2}E^{14}_{7,12}
-\frac{1}{2}E^{14}_{8,11}
\edm\bdm
\vrho(A_{13})=
-\frac{1}{2}E^{14}_{4,13}
+\frac{\sqrt{3}}{2}E^{14}_{4,14}
-\frac{1}{2}E^{14}_{5,11}
+\frac{1}{2}E^{14}_{6,12}
+\frac{1}{2}E^{14}_{7,9}
+\frac{1}{2}E^{14}_{8,10}
\edm\bdm
\vrho(A_{14})=
\frac{1}{2}E^{14}_{3,13}
-\frac{\sqrt{3}}{2}E^{14}_{3,14}
-\frac{1}{2}E^{14}_{5,12}
-\frac{1}{2}E^{14}_{6,11}
-\frac{1}{2}E^{14}_{7,10}
+\frac{1}{2}E^{14}_{8,9}
\edm\bdm
\vrho(A_{15})=
+\frac{1}{2}E^{14}_{1,11}
-\frac{1}{2}E^{14}_{2,12}
-\frac{1}{2}E^{14}_{3,9}
+\frac{1}{2}E^{14}_{4,10}
+\frac{1}{2}E^{14}_{8,13}
+\frac{\sqrt{3}}{2}E^{14}_{8,14}
\edm\bdm
\vrho(A_{16})=
+\frac{1}{2}E^{14}_{1,12}
+\frac{1}{2}E^{14}_{2,11}
-\frac{1}{2}E^{14}_{3,10}
-\frac{1}{2}E^{14}_{4,9}
-\frac{1}{2}E^{14}_{7,13}
-\frac{\sqrt{3}}{2}E^{14}_{7,14}
\edm\bdm
\vrho(A_{17})=
+\frac{1}{2}E^{14}_{1,9}
+\frac{1}{2}E^{14}_{2,10}
+\frac{1}{2}E^{14}_{3,11}
+\frac{1}{2}E^{14}_{4,12}
+\frac{1}{2}E^{14}_{6,13}
+\frac{\sqrt{3}}{2}E^{14}_{6,14}
\edm\bdm
\vrho(A_{18})=
-\frac{1}{2}E^{14}_{1,10}
+\frac{1}{2}E^{14}_{2,9}
-\frac{1}{2}E^{14}_{3,12}
+\frac{1}{2}E^{14}_{4,11}
-\frac{1}{2}E^{14}_{5,13}
-\frac{\sqrt{3}}{2}E^{14}_{5,14}
\edm\bdm
\vrho(A_{19})=
+\frac{\sqrt{3}}{2}E^{14}_{1,3}
-\frac{\sqrt{3}}{2}E^{14}_{2,4}
+\frac{\sqrt{3}}{2}E^{14}_{5,7}
-\frac{\sqrt{3}}{2}E^{14}_{6,8}
\edm\bdm
\vrho(A_{20})=
+\frac{\sqrt{3}}{2}E^{14}_{1,4}
+\frac{\sqrt{3}}{2}E^{14}_{2,3}
+\frac{\sqrt{3}}{2}E^{14}_{5,8}
+\frac{\sqrt{3}}{2}E^{14}_{6,7}
\edm\bdm
\vrho(A_{21})=
-\frac{\sqrt{3}}{2}E^{14}_{1,2}
-\frac{\sqrt{3}}{2}E^{14}_{3,4}
-\frac{\sqrt{3}}{2}E^{14}_{5,6}
-\frac{\sqrt{3}}{2}E^{14}_{7,8}
\edm
%
%
\subsection{\boldmath$SU(4)/\SO(2)$}\label{ch:ap:su4}$ $
\smallskip\\
Looking at the given embedding, we define 
\bdm
K^1_{1}:=\frac{1}{\sqrt{2\alpha}}E^4_{1,3},\quad
K^1_{2}:=\frac{i}{\sqrt{2\alpha}}S^4_{1,3}, \quad
K^1_{3}:=\frac{1}{\sqrt{2\alpha_{2}}}E^4_{2,4},\quad
K^1_{4}:=\frac{i}{\sqrt{2\alpha_{2}}}S^4_{2,4},
\edm\bdm
K^1_{5}:=\frac{1}{\sqrt{2\alpha_{3}}}E^4_{2,3},\quad
K^1_{6}:=\frac{i}{\sqrt{2\alpha_{3}}}S^4_{2,3},\quad
K^1_{7}:=\frac{1}{\sqrt{2\alpha_{4}}}E^4_{1,4}, \quad
K^1_{8}:=\frac{i}{\sqrt{2\alpha_{4}}}S^4_{1,4},
\edm\bdm
K^1_{9}:=\frac{1}{\sqrt{2\alpha_{5}}}E^4_{1,2},\quad 
K^1_{10}:=\frac{i}{\sqrt{2\alpha_{6}}}S^4_{1,2},\quad
K^1_{11}:=\frac{}{\sqrt{2\alpha_{7}}}E^4_{3,4}, \quad
K^1_{12}:=\frac{i}{\sqrt{2\alpha_{8}}}S^4_{3,4},
\edm\bdm
K^1_{13}:=\frac{i}{2\sqrt{\beta}}(S^4_{1,1}-S^4_{2,2}+S^4_{3,3}-S^4_{4,4}),\quad
 K^1_{14}:=\frac{i}{2\sqrt{\gamma}}(-S^4_{1,1}+S^4_{2,2}+S^4_{3,3}-S^4_{4,4})
\edm
and
\bdm
H^1:=\frac{i}{2}(S^4_{1,1}+S^4_{2,2}-S^4_{3,3}-S^4_{4,4}).
\edm
We have $\su(4)=\so(2)\oplus\m_1=span(\{H^1\} \cup \{K^1_i~|~i=1..14\})$.
We get the representation of $\SO(2)$ as
\bdm
\vrho_1(H^1)K^1_i=K^1_{i+1},~ \vrho_1(H^1)K^1_{i+1}=-K^1_i \mbox{ for } i = 1, 3, 5, 7
\edm
and
\bdm
\vrho_1(H^1)K^1_i=0 \mbox{ for } i = 9..14.
\edm
This gives an identification $\m_1\ra\m$, $K^1_i\mapsto B_i$ inducing an 
inclusion $\SO(2)\subset\Sp(3)\subset\SO(\m)$ because of 
$\vrho_1(H^1)=\sqrt{2}\vrho(A_{21})$, and therefore defines an $\Sp(3)$  
structure on $\SU(4)/\SO(2)$.  
We compute the torsion and get in the basis we just defined
\begin{align*}
T & =\frac{1}{\sqrt{2\alpha}}(  e_{1} e_{5} e_{9}
-e_{1} e_{6} e_{10}
+e_{1}  e_{7}  e_{11} 
+e_{1}  e_{8}  e_{12}
+e_{2}  e_{5}  e_{10} 
+e_{2}  e_{6}  e_{9}
-e_{2}  e_{7}  e_{12} 
+e_{2}  e_{8}  e_{11}\\
& -e_{3}  e_{5}  e_{11 }
+e_{3}  e_{6}  e_{12}
-e_{3}  e_{7}  e_{9} 
-e_{3}  e_{8}  e_{10}
-e_{4}  e_{5}  e_{12} 
-e_{4}  e_{6}  e_{11}
+e_{4}  e_{7}  e_{10} 
-e_{4}  e_{8}  e_{9}) \\
&+\frac{\sqrt{\beta}}{\alpha}(e_{5} e_{6}  e_{13}
-e_{7}  e_{8}  e_{13} 
-e_{9}  e_{10}  e_{13}
-e_{11}  e_{12}  e_{13}) 
+\frac{\sqrt{\gamma}}{\alpha}(e_{1}  e_{2}  e_{14}
-e_{3}  e_{4}  e_{14} 
+e_{9}  e_{10}  e_{14}
-e_{11}  e_{12}  e_{14}). 
\end{align*}
\begin{NB}
This is not the only possible inclusion $\SO(2)\subset\Sp(3)$. We get other 
identifications $\m_1\cong\m$ inducing other $\Sp(3)$ structures.
\end{NB}
%
\subsection{\boldmath$\U(4)/\SO(2)\times\SO(2)$}\label{ch:ap:u4}$ $ 
\smallskip\\
We define a basis using almost the same matrices as above but taking other normalizers
\bdm
K^2_{1}:=\frac{1}{\sqrt{2\alpha}}E^4_{1,3}, \quad 
K^2_{2}:=\frac{i}{\sqrt{2\alpha}}S^4_{1,3}, \quad  
K^2_{3}:=\frac{1}{\sqrt{2\alpha_{2}}}E^4_{2,4},\quad 
K^2_{4}:=\frac{i}{\sqrt{2\alpha_{2}}}S^4_{2,4},\edm\bdm
K^2_{5}:=\frac{1}{\sqrt{2\alpha_{3}}}E^4_{2,3}, \quad 
K^2_{6}:=\frac{i}{\sqrt{2\alpha_{3}}}S^4_{2,3},\quad 
K^2_{7}:=\frac{1}{\sqrt{2\alpha_{4}}}E^4_{1,4}, \quad 
K^2_{8}:=\frac{i}{\sqrt{2\alpha_{4}}}S^4_{1,4},\edm\bdm
K^2_{9}:=\frac{1}{\sqrt{2\alpha_{5}}}E^4_{1,2},\quad 
K^2_{10}:=\frac{i}{\sqrt{2\alpha_{5}}}S^4_{1,2},\quad 
K^2_{11}:=\frac{}{\sqrt{2\alpha_{6}}}E^4_{3,4}, \quad 
K^2_{12}:=\frac{i}{\sqrt{2\alpha_{6}}}S^4_{3,4},\edm\bdm
K^2_{13}:=\frac{i}{2\sqrt{\beta}}(S^4_{1,1}-S^4_{2,2}+S^4_{3,3}-S^4_{4,4}), \quad 
K^2_{14}:=\frac{i}{2\sqrt{\gamma}}(S^4_{1,1}+S^4_{2,2}+S^4_{3,3}+S^4_{4,4})
\edm
and
\bdm
H^2_1:=\frac{i}{2}(S^4_{1,1}+S^4_{2,2}-S^4_{3,3}-S^4_{4,4}),\quad
H^2_2:=\frac{i}{2}(-S^4_{1,1}+S^4_{2,2}+S^4_{3,3}-S^4_{4,4}),
\edm
getting $\un(4)=\so(2)\oplus\so(2)\oplus\m_2=span(\{H^2_1,~H^2_2\}\cup \{K^2_i~|~i=1..14\})$.
The representation of $\SO(2)\times \SO(2)$ is given by
\bdm
\vrho_2(H^2_1)K^2_i=K^2_{i+1},~ \vrho_2(H^2_1)K^2_{i+1}=-K^2_i 
\mbox{ for } i = 1, 3, 5, 7, \quad \vrho_2(H^1)K^1_i=0 \mbox{ for } i = 9..14,
\edm
and
\bdm
\vrho_2(H^2_2)K^2_i=-K^2_{i+1},~ \vrho_2(H^2_2)K^2_{i+1}=K^2_i \mbox{ for } i = 1,9,
\edm
\bdm
\vrho_2(H^2_2)K^2_i=K^2_{i+1},~ \vrho_2(H^2_2)K^2_{i+1}=-K^2_i \mbox{ for } i = 3,11,
\edm
\bdm
\vrho_2(H^1)K^2_i=0 \mbox{ for } i = 5..8,13,14.
\edm
We choose the identification $\m_2\ra\m$, $K^2_i\mapsto B_i$ inducing a 
inclusion $\SO(2)\times\SO(2)\subset\Sp(3)\subset\SO(\m)$ because of 
$\vrho_2(H^2_1)=\sqrt{2}\vrho(A_{21})$ and
$\vrho_2(H^2_2)=\sqrt{2}\vrho(A_{10})$, 
therefore defining a $\Sp(3)$ structure on $\SU(4)/(\SO(2)\times\SO(2))$.
In this basis we can compute the torsion and get
\begin{align*}
T&=\frac{1}{\sqrt{2\alpha}}(e_{1}  e_{5}  e_{9}
-e_{1}  e_{6}  e_{10}
+e_{1}  e_{7}  e_{11}
+e_{1}  e_{8}  e_{12}
+e_{2}  e_{5}  e_{10}
+e_{2}  e_{6}  e_{9}
-e_{2}  e_{7}  e_{12}
+e_{2}  e_{8}  e_{11}\\
& -e_{3}  e_{5}  e_{11}
+e_{3}  e_{6}  e_{12}
-e_{3}  e_{7}  e_{9}
-e_{3}  e_{8}  e_{10}
-e_{4}  e_{5}  e_{12}
-e_{4}  e_{6}  e_{11}
+e_{4}  e_{7}  e_{10}
-e_{4}  e_{8}  e_{9})\\
&+\frac{\sqrt{\beta}}{\alpha}(e_{5}  e_{6}  e_{13}
-e_{7}  e_{8}  e_{13}
-e_{9}  e_{10}  e_{13}
-e_{11}  e_{12}  e_{13})
\end{align*}
%
%
\subsection{\boldmath$\U(4)\times\U(1)/\SO(2)\times\SO(2)\times\SO(2)$}\label{ch:ap:u4u1}$
$ 
\smallskip\\
We define a basis of 
$\un(4)\oplus\un(1)=\so(2)\oplus\so(2)\oplus\so(2)\oplus\m_3$ 
with $K^3_i$ for $i=1..14$ a basis of $\m_3$ and $H^3_i$  
for $i=1..3$ a basis of $\so(2)\oplus\so(2)\oplus\so(2)$:
\bdm
K^3_i:=(K^2_i,0) \mbox{ for } i\neq 13 \mbox { and } 
K_{13}^3:=(0,\frac{i}{\sqrt{\beta}}),
\edm
\bdm
H^3_i:=(H^2_i,0) \mbox{ for } i=1,2 \mbox{ and } 
H^3_3:=(\frac{i}{2}(S^4_{1,1}-S^4_{2,2}+S^4_{3,3}-S^4_{4,4}),0).
\edm
Identifying $\m\cong\m_1\cong\m_2\cong\m_3$ with 
$A_i\mapsto K^1_i\mapsto K^2_i\mapsto K^3_i$ we get the representation 
$\vrho_3$ of $\so(2)\oplus\so(2)\oplus\so(2)$ by
\bdm
\sqrt{2}\vrho(A_{21})=\vrho_1(H^1)=\vrho_2(H^2_1)=\vrho_3(H^3_1), \quad 
\sqrt{2}\vrho(A_{10})=\vrho_2(H^2_2)=\vrho_3(H^3_2), \quad \sqrt{2}\vrho(A_{9})=\vrho_3(H^3_3)
\edm
and again we get a $\Sp(3)$ structure. In this basis we can compute the torsion and get
\begin{align*}
T=&\frac{1}{\sqrt{2\alpha}}(e_{1}  e_{5}  e_{9}
-e_{1}  e_{6}  e_{10}
+e_{1}  e_{7}  e_{11}
+e_{1}  e_{8}  e_{12}
+e_{2}  e_{5}  e_{10}
+e_{2}  e_{6}  e_{9}
-e_{2}  e_{7}  e_{12}
+e_{2}  e_{8}  e_{11}\\
&-e_{3}  e_{5}  e_{11}
+e_{3}  e_{6}  e_{12}
-e_{3}  e_{7}  e_{9}
-e_{3}  e_{8}  e_{10}
-e_{4}  e_{5}  e_{12}
-e_{4}  e_{6}  e_{11}
+e_{4}  e_{7}  e_{10}
-e_{4}  e_{8}  e_{9}).
\end{align*}
%
\subsection{\boldmath$\SU(5)/\Sp(2)$}\label{ch:ApSU5}$ $
\smallskip\\
The Lie algebra of $\Sp(2)\subset\Sp(3)$ and its splitting of $V^{14}$ 
is given by Theorem \ref{thm:maxsub}. Calculating the torsion tensor we get
\begin{align*}
T=
&\frac{2\alpha-\beta}{2\alpha\sqrt{\beta}}(e_{1}  e_{2}  e_{13}
+e_{1}  e_{5}  e_{9}
-e_{1}  e_{6}  e_{10}
+e_{1}  e_{7}  e_{11}
+e_{1}  e_{8}  e_{12}
+e_{2}  e_{5}  e_{10}
+e_{2}  e_{6}  e_{9}
-e_{2}  e_{7}  e_{12}\\
& +e_{2}  e_{8}  e_{11}
+e_{3}  e_{4}  e_{13}
+e_{3}  e_{5}  e_{11}
-e_{3}  e_{6}  e_{12}
-e_{3}  e_{7}  e_{9}
-e_{3}  e_{8}  e_{10}
+e_{4}  e_{5}  e_{12}\\
&+e_{4}  e_{6}  e_{11}
+e_{4}  e_{7}  e_{10}
-e_{4}  e_{8}  e_{9}
-e_{5}  e_{6}  e_{13}
-e_{7}  e_{8}  e_{13})\\
&+\frac{\sqrt{5\beta\gamma}-\sqrt{6}(\alpha+\beta)}{2\alpha\sqrt{\beta}}(e_{1}  e_{2}  e_{14}
+e_{3}  e_{4}  e_{14}
+e_{5}  e_{6}  e_{14}
+e_{7}  e_{8}  e_{14}).
\end{align*}

\subsection{Maximal subgroups of \boldmath$\Sp(3)$}
\label{ch:ApSubGroups}
%
Using Dynkin's results \cite{dynk}, Gorodski and Podesta listed the maximal  
connected subgroups of $\Sp(n)$ in \cite{GP}. We restate the result for 
$G\subset\Sp(3)$ and add the decompositions of $V^{14}$ into
subrepresentations  
of $G\subset\Sp(3)$, computed easily via an appropriate computer 
algebra system.

Given a group $G\subset\Sp(3)$, we give a basis of its Lie algebra 
$\g\subset \sp(3)=\langle A_i~|~i=1..21\rangle$, the decomposition of 
$V^{14}=V_1\oplus..\oplus V_r$ in irreducible subspaces, and a basis of 
each $V_k\subset V^{14}=\langle B_i~|~i=1..14\rangle$.
\begin{thm}\label{thm:maxsub}
All maximal connected subgroups of $\Sp(3)$ and 
the decomposition of $V^{14}$ into submodules for these subgroups
are listed in Table $2$.
Furthermore, the subgroup $\Sp(2)\subset\Sp(2)\times\Sp(1)\subset\Sp(3)$ 
with Lie algebra  $\sp(2)=\langle \{A_i~|~i=1..10\}\rangle$ 
acts irreducibly on  $\Delta_5$, the irreducible $8$-dimensional 
spin representation of
$\Spin(5)\cong\Sp(2)$ and on $\p^5$, its usual vector representation, 
and thus $V^{14}$ has the same decomposition into $\Sp(2)$-isotopic summands as
under $\Sp(2)\times\Sp(1)$,
\bdm
V^{14}\ \stackrel{\Sp(2)}{=}\  \Delta_5\, 
\oplus\, \p^5\, \oplus\, \p^1,
\edm
$\p^1$ being the trivial representation.
\end{thm}

\begin{table}
{\small \hspace{-.4cm}
\begin{tabular}{l|p{2cm}|p{4.3cm}|p{1.4cm}|p{5.3cm}|}
\cline{2-5}
&$G\subset \Sp(3)$ & Basis of $\g\subset\sp(3)$ & $V^{14}$& Basis of invariant submodules\\
\cline{2-5}
\cline{2-5}
\multirow{2}{.22cm}{{}}&\multirow{2}{2.1cm}{$ \U(3)$} & 
\multirow{2}{4.3cm}{$A_1,A_2,A_9,A_{10},A_{11},A_{12},$\newline$A_{17},A_{18},A_{21}$
}& $\R^8$ & $B_1,B_2,B_5,B_6,B_9,B_{10},B_{13},B_{14}$\\
\cline{4-5}
&&& $\R^6$ & $B_3,B_4,B_7,B_8,B_{11},B_{12}$\\
\cline{2-5}
\multirow{6}{.1cm}{}&\multirow{6}{2.1cm}{$\SO(3)$} & & 
\multirow{3}{1.4cm}{$\R^9$} & $\frac{2}{\sqrt{3}}B_{13}+B_{14},-\sqrt{\frac{5}{2}}B_{6}+B_{10},$\\
&& $\sqrt{10}A_1+4A_{17}-3A_{19},$ &&
$-\sqrt{\frac{5}{2}}B_{5}+B_{9}, \frac{3}{\sqrt{5}}B_{2}+B_{8},$\\
&&$\sqrt{10}A_2+4A_{18}+3A_{20},$&&$-\frac{3}{\sqrt{5}}B_{1}+B_{7},B_3,B_4,B_{11},B_{12}$\\
\cline{4-5}
&&
$3A_{9}+5A_{10}+A_{21}$&\multirow{3}{1.4cm}{$\R^5$}&$-\frac{\sqrt{3}}{2}B_{13}+B_{14},
\sqrt{\frac{2}{5}}B_{6}+B_{10},$\\
&&&&$\sqrt{\frac{2}{5}}B_{5}+B_{9}, -\frac{\sqrt{5}}{3}B_{2}+B_{8},$\\
&&&&$\frac{\sqrt{5}}{3}B_{1}+B_{7}$\\
\cline{2-5}
\multirow{3}{.1cm}{}&\multirow{3}{2.1cm}{$\Sp(2)\times\Sp(1)$} & 
\multirow{3}{4.3cm}{$A_1,~...~,A_{10},~A_{19},A_{20},A_{21}$}& $\Delta_5$& $B_1, .. ,B_8$\\
\cline{4-5}
&&&$\p^5$&$B_9,..B_{13}$\\
\cline{4-5}
&&&$\p^1$ & $B_{14}$\\
\cline{2-5}
\multirow{2}{.1cm}{}&\multirow{2}{2.1cm}{$\SO(3)\times\Sp(1)$} & 
\multirow{2}{4.3cm}{$A_1,A_{11},A_{17}, A_9+A_{10}+A_{21},$ 
$A_{5}+A_{6}+A_{19},A_{7}+A_{8}+A_{20}$} &$\R^3\otimes\R^3$
&$B_1,B_3,B_4,B_5,B_7,B_8,B_9,B_{11},B_{12}$\\
\cline{4-5}
&&&$\R^5\otimes\R^1$&$B_2,B_6,B_{10},B_{13},B_{14}$\\
\cline{2-5}
\end{tabular}
}
\caption{Maximal connected subgroups of $\Sp(3)$ and decompositions 
of $V^{14}$ into submodules.} 
\end{table}
%

%
\vspace{4cm}
    
\vspace{2cm}
\end{document}